\documentclass[a4paper,oneside,11pt]{article}%
\usepackage{makeidx}
\usepackage[english]{babel}
\usepackage{amsmath}
\usepackage{amsfonts}
\usepackage{amssymb}
\usepackage{stmaryrd}
\usepackage{graphicx}
\usepackage{mathrsfs}
\usepackage[colorlinks,linkcolor=red,anchorcolor=blue,citecolor=blue,urlcolor=blue]{hyperref}
\usepackage[symbol*,ragged]{footmisc}

\allowdisplaybreaks[4]
\providecommand{\U}[1]{\protect\rule{.1in}{.1in}}
\providecommand{\U}[1]{\protect \rule{.1in}{.1in}}

\newtheorem{theorem}{Theorem}[section]

\newtheorem{lemma}[theorem]{Lemma}

\newenvironment{proof}[1][Proof]{\noindent \textbf{#1.} }{\  \rule{0.5em}{0.5em}}
\numberwithin{equation}{section}

\begin{document}

\title{An Additive Problem over Piatetski--Shapiro \\ Primes and Almost--primes}

\author{Jinjiang Li\footnotemark[1] \,\,\,\,\,  \& \,\,\, Min Zhang\footnotemark[2]
                    \vspace*{-4mm} \\
     $\textrm{\small Department of Mathematics, China University of Mining and Technology\footnotemark[1]}$
                    \vspace*{-4mm} \\
     \small  Beijing 100083, P. R. China
                     \vspace*{-4mm}  \\
     $\textrm{\small School of Applied Science, Beijing Information Science and Technology University\footnotemark[2]}$
                    \vspace*{-4mm}  \\
     \small  Beijing 100192, P. R. China }

\footnotetext[2]{Corresponding author. \\
    \quad\,\, \textit{ E-mail addresses}:
     \href{mailto:jinjiang.li.math@gmail.com}{jinjiang.li.math@gmail.com} (J. Li),
     \href{mailto:min.zhang.math@gmail.com}{min.zhang.math@gmail.com} (M. Zhang).  }

\date{}
\maketitle

{\textbf{Abstract}}: Let $\mathcal{P}_r$ denote an almost--prime with at most $r$ prime factors, counted
according to multiplicity. In this paper, we establish a theorem of Bombieri--Vinogradov type for the Piatetski--Shapiro
primes $p=[n^{1/\gamma}]$ with $\frac{85}{86}<\gamma<1$. Moreover, we use this result to prove that, for $0.9989445<\gamma<1$,
there exist infinitely many Piatetski--Shapiro primes such that $p+2=\mathcal{P}_3$, which improves the previous results of
Lu \cite{Lu-2018}, Wang and Cai \cite{Wang-Cai-2011}, and Peneva \cite{Peneva-2003}.


{\textbf{MR(2020) Subject Classification}}: 11L07, 11L20, 11P32, 11N36

\section{Introduction and main result}
The ternary Goldbach problem asserts that every odd integer $n\geqslant9$ can be represented in the form
\begin{equation}\label{three-prime}
   n=p_1+p_2+p_3,
\end{equation}
where $p_1,p_2,p_3$ are odd prime numbers. In 1937, Vinogradov \cite{Vinogradov-1937} proved that a representation of the type (\ref{three-prime}) exists for every sufficiently large odd integer. The binary Goldbach problem, which states that every even integer $N\geqslant6$ can be written as the sum of two odd primes, also remains unsettled. Another central problem in the theory of prime distribution, namely the twin prime conjecture, states that there exist infinitely many primes $p$ such that $p+2$ is also prime. Although the conjecture has resisted all attacks, there have been spectacular partial achievements.
One well known result is due to Chen \cite{Chen-1966,Chen-1973}, who proved that there exist infinitely many primes $p$ such that $p+2$ has at most $2$ prime factors.

An important approach for studying the binary Goldbach problem is by the use of sieve methods. As usual, we denote
by $\mathcal{P}_r$ an almost--prime with with at most $r$ prime factors, counted according to multiplicity.
In 1947, R\'{e}nyi \cite{Renyi-1947} was the first to prove that there exists an $r$ such that every sufficiently large even integer $N$ is representable in the form
\begin{equation}\label{binary-Goldbach}
   N=p+\mathcal{P}_r,
\end{equation}
where $p$ is a prime number. The best result in this direction is due to Chen \cite{Chen-1966,Chen-1973} who
showed that (\ref{binary-Goldbach}) holds for $r=2$.

Let $\gamma$ be a real number such that $\frac{1}{2}<\gamma<1$. Define
\begin{equation*}
   \pi_\gamma(x):=\#\big\{p\leqslant x: p=[n^{1/\gamma}]\,\,\, \textrm{for some $n\in\mathbb{N}$}  \big\},
\end{equation*}
In 1953, Piatetski--Shapiro \cite{Piatetski-Shapiro-1953} showed that
\begin{equation*}
   \pi_\gamma(x)\sim\frac{x^\gamma}{\log x},\qquad (x\to \infty),
\end{equation*}
for $\frac{11}{12}<\gamma<1$. The prime numbers of the form $p=[n^{1/\gamma}]$ are called \textit{Piatetski--Shapiro primes of type $\gamma$}. Since then, by using the close connection between the lower bound for $\gamma$ and the estimates of the exponential sums over primes, this range for $\gamma$ has been enlarged by a number of authors
\cite{BHR-1995,Heath-Brown-1983,Jia-1993,Jia-1994,Kolesnik-1972,Kolesnik-1985,Kumchev-1999,Leitmann-1980,Liu-Rivat-1992,Rivat-1992}.
The best results are given by Rivat and Sargos \cite{Rivat-Sargos-2001} and Rivat and Wu \cite{Rivat-Wu-2001}, where it is proved that
\begin{equation*}
   \pi_\gamma(x)\sim\frac{x^\gamma}{\log x}
\end{equation*}
for $\frac{2426}{2817}<\gamma<1$, and
\begin{equation*}
   \pi_\gamma(x)\gg\frac{x^\gamma}{\log x}
\end{equation*}
for $\frac{205}{243}<\gamma<1$, respectively.

In 1992, Balog and Friedlander \cite{Balog-Friedlander-1992} found an asymptotic formula for the number of solutions of the equation (\ref{three-prime}) with variables restricted to the Piatetski--Shapiro primes. An interesting corollary of their theorem is that every sufficiently large odd integer can be written as the sum of two primes and a Piatetski--Shapiro
prime of type $\gamma$, provided that $\frac{8}{9}<\gamma<1$. Afterwards, their studies in this direction were subsequently continued by Jia \cite{Jia-1995} and by Kumchev \cite{Kumchev-1997}, and generalized by Cui \cite{Cui-2004} and Li and Zhang \cite{Li-Zhang-2018}, consecutively and respectively.

Based on the above results, it is interesting to investigate the solvability of the equation (\ref{binary-Goldbach}) when $p$ is a Piatetski--Shapiro prime. It is naturally expected that a theorem of Bombieri--Vinogradov type holds for the Piatetski--Shapiro primes. In the early days, the only result in this direction, due to Leitmann \cite{Leitmann-1977}, gives a very low level of distribution which does not allow us to determine the value of the parameter $r$.

In 2003, Peneva \cite{Peneva-2003} obtained a mean value theorem of Bombieri--Vinigradov's type for Piatetski--Shapiro primes, by which and sieve methods she showed that, for every sufficiently large even integer $N$, (\ref{binary-Goldbach}) is solvable with $p=[n^{1/\gamma}]$ a Piatetski--Shapiro prime, and $r$ is the least positive integer satisfying the inequality
\begin{equation*}
  r+1-\frac{\log\frac{4}{1+3^{-r}}}{\log 3}>\frac{1}{\xi(\gamma)}+\varepsilon,
\end{equation*}
where
\begin{equation}\label{level-1}
  \xi=\xi(\gamma)=
  \begin{cases}
    \frac{755}{424}\gamma-\frac{331}{212}-\varepsilon, & \textrm{for $\frac{662}{755}<\gamma\leqslant\frac{608}{675}$},\\
    \frac{5}{4}\gamma-\frac{13}{12}-\varepsilon, & \textrm{for $\frac{608}{675}<\gamma<1$}.
  \end{cases}
\end{equation}
By using the above level $\xi$, Peneva \cite{Peneva-2003} proved that (\ref{binary-Goldbach}) is solvable for $r=7$ with
a Piatetski--Shapiro prime $p=[n^{1/\gamma}]$ and $0.9854<\gamma<1$. Essentially, from the arguments similar to that in Peneva \cite{Peneva-2003}, one can obtain that, there exist infinitely many Piatetski--Shapiro primes of type $\gamma$ such that $p+2=\mathcal{P}_7$ with $0.9854<\gamma<1$.

In 2011, by using the same level $\xi$ in (\ref{level-1}), Wang and Cai \cite{Wang-Cai-2011} improved the result
of Peneva \cite{Peneva-2003}, and showed that there exist infinitely many Piatetski--Shapiro primes of type $\gamma$ such that $p+2=\mathcal{P}_5$ with $\frac{29}{30}<\gamma<1$. Afterwards, Lu \cite{Lu-2018}, in 2018, reestablish a mean value theorem of Bombieri--Vinigradov's type with level $\xi=\xi(\gamma)=(13\gamma-12)/4-\varepsilon$ for $\frac{12}{13}<\gamma<1$. By using this level, Lu \cite{Lu-2018} strengthen the result of Wang and Cai \cite{Wang-Cai-2011}. He proved that there exist infinitely many Piatetski--Shapiro primes of type $\gamma$ such that $p+2=\mathcal{P}_4$ with $0.9993<\gamma<1$.

In this paper, we shall continue to improve the result of Lu \cite{Lu-2018}, and establish
the two following theorems.

\begin{theorem}\label{Theorem-1}
  Suppose that $\gamma$ is a real number satisfying $\frac{85}{86}<\gamma<1,\,a\not=0$ is a fixed integer. Then for any given constant $A>0$ and any sufficiently small $\varepsilon>0$, there holds
\begin{equation}\label{Thm-eq}
   \sum_{\substack{d\leqslant x^\xi\\ (d,a)=1}}\Bigg|\sum_{\substack{p\leqslant x\\ p\equiv a\!\!\!\!\!\pmod d\\
          p=[k^{1/\gamma}]}}1-\frac{1}{\varphi(d)}\pi_\gamma(x)\Bigg|\ll\frac{x^\gamma}{(\log x)^A},
\end{equation}
where
\begin{equation*}
   \xi=\xi(\gamma)=\frac{129}{4}\gamma-\frac{255}{8}-\varepsilon;
\end{equation*}
the implied constant in (\ref{Thm-eq}) depends only on $A$ and $\varepsilon$.
\end{theorem}

\begin{theorem}\label{Theorem-2}
Suppose that $\gamma$ is a real number satisfying $0.9989445<\gamma<1$. Then there exist infinitely many
Piatetski--Shapiro primes of type $\gamma$ such that
\begin{equation*}
   p+2=\mathcal{P}_3.
\end{equation*}
\end{theorem}

\noindent
\textbf{Remark.} The key point of improving the number $r$ such that $p+2=\mathcal{P}_r$ with Piatetski--Shapiro prime
$p=[n^{1/\gamma}]$  is to enlarge the level $\xi=\xi(\gamma)$, for $\gamma$ near to $1$, of the mean value theorem of Bombieri--Vinigradov's type for Piatetski--Shapiro primes. In order to compare our result with the results of Lu \cite{Lu-2018} and Peneva \cite{Peneva-2003}, we list the numerical result as follows:
\begin{align*}
  \xi(\gamma)=&\,\, \frac{129}{4}\gamma-\frac{255}{8}-\varepsilon\to\frac{3}{8}=0.375, \quad\,\,\,\textrm{for $\gamma\to1$},   \\
  \xi(\gamma)=&\,\, \frac{13\gamma-12}{4}-\varepsilon\to \frac{1}{4}=0.25, \qquad\quad\,\, \textrm{for $\gamma\to1$},    \\
  \xi(\gamma)=&\,\, \frac{5}{4}\gamma-\frac{13}{12}-\varepsilon\to\frac{1}{6}=0.1666\cdots, \quad\textrm{for $\gamma\to1$}.
\end{align*}
In order to establish Theorem , we employ the method of Vaughan \cite{Vaughan-1976}, combining with the weighted sieve of Richert and the method of Chen \cite{Chen-1973}.

\noindent
\textbf{Notation.} Throughout this paper, $x$ is a sufficiently large number; $\varepsilon$ and $\eta$ are sufficiently small
positive numbers, which may be different in each occurrences. Let $p$, with or without subscripts, always denote a prime number. We use $[x],\,\{x\}$ and $\|x\|$ to denote the integral part of $x$, the fractional part of $x$
and the distance from $x$ to the nearest integer, respectively. As usual, $\varphi(n),\Lambda(n),\tau(n)$ and $\mu(n)$ denote
Euler's function, von Mangoldt's function, the Dirichlet divisor function and M\"{o}bius' function, respectively.
Also, we use $\chi\bmod q$ to denote a Dirichlet character modulo $q$, and $\chi^0\bmod q$ the principal character. Especially, we use $\Sigma^*$ to denote sums over all primitive characters. Let $(m_1,m_2,\dots,m_k)$ and $[m_1,m_2,\dots,m_k]$ be the greatest common divisor and the least common multiple of $m_1,m_2,\dots,m_k$, respectively. We write $L=\log x$; $e(t)=\exp(2\pi it)$; $\psi(t)=t-[t]-\frac{1}{2}$. The notation $n\sim X$ means that $n$ runs through a subinterval
of $(X,2X]$, whose endpoints are not necessarily the same in the different occurrences and may depend on the outer summation variables. $f(x)\ll g(x)$ means that $f(x)=O(g(x))$; $f(x)\asymp g(x)$ means that $f(x)\ll g(x)\ll f(x)$.

\section{Preliminaries}
In this section, we shall reduce the problem of estimating the sum in (\ref{Thm-eq}) to estimating exponential sums over primes.

For $1/2<\gamma<1$, it is easy to see that
\begin{equation*}
  [-k^\gamma]-[-(k+1)^\gamma]=
  \begin{cases}
    1, & \textrm{if\,\,$k=[n^{1/\gamma}]$},\\
    0, & \textrm{otherwise}.
  \end{cases}
\end{equation*}
For convenience, we put $D=x^\xi$. In order to prove (\ref{Thm-eq}), it is sufficient to prove that
\begin{equation}\label{suffi-1}
 \sum_{\substack{d\leqslant D\\ (a,d)=1}}\Bigg|\sum_{\substack{n\leqslant x\\ n\equiv a\!\!\!\!\!\pmod d}}\Lambda(n)
 \big((n+1)^\gamma-n^\gamma\big)-\frac{1}{\varphi(d)}\sum_{n\leqslant x}\Lambda(n)\big((n+1)^\gamma-n^\gamma\big)\Bigg|
 \ll x^\gamma L^{-A},
\end{equation}
\begin{equation}\label{suffi-2}
\sum_{\substack{d\leqslant D\\ (a,d)=1}}\Bigg|\sum_{\substack{n\leqslant x\\ n\equiv a\!\!\!\!\!\pmod d}}\Lambda(n)
 \big(\psi\big(-n^\gamma\big)-\psi\big(-(n+1)^\gamma\big)\big)\Bigg|
 \ll x^\gamma L^{-A}
\end{equation}
and
\begin{equation}\label{suffi-3}
\sum_{\substack{d\leqslant D\\ (a,d)=1}}\frac{1}{\varphi(d)}
\Bigg|\sum_{n\leqslant x}\Lambda(n)\big(\psi\big(-n^\gamma\big)-\psi\big(-(n+1)^\gamma\big)\big)\Bigg|\ll x^\gamma L^{-A}.
\end{equation}
The estimate (\ref{suffi-1}) can be obtained from the Bombieri--Vinogradov theorem by using partial summation and it holds for every $\gamma\in(1/2,1)$ and $D=x^{1/2-\varepsilon}$, where $\varepsilon>0$ is sufficiently small. The estimate (\ref{suffi-3}) follows from the arguments in \cite{Heath-Brown-1983}. Thus, we only have to prove (\ref{suffi-2}). Obviously, (\ref{suffi-2}) will follow, if we can prove that for $X\leqslant x$, there holds
\begin{equation}\label{suffi-4}
\sum_{\substack{d\leqslant D\\ (a,d)=1}}\Bigg|\sum_{\substack{n\sim X\\ n\equiv a\!\!\!\!\!\pmod d}}\Lambda(n)
 \big(\psi\big(-n^\gamma\big)-\psi\big(-(n+1)^\gamma\big)\big)\Bigg| \ll x^\gamma L^{-A}.
\end{equation}
Let $\eta>0$ be a sufficiently small number. If $X\leqslant x^{1-\eta}$, then the left--hand side of (\ref{suffi-4}) is
\begin{align*}
\ll & \,\, \sum_{\substack{d\leqslant D\\ (a,d)=1}}\Bigg|\sum_{\substack{n\sim X\\ n\equiv a\!\!\!\!\!\pmod d}}\Lambda(n)
           \big((n+1)^\gamma-n^\gamma\big)\Bigg|
                   \nonumber \\
   & \,\,  +\sum_{\substack{d\leqslant D\\ (a,d)=1}}
           \Bigg|\sum_{\substack{n\sim X\\ n\equiv a\!\!\!\!\!\pmod d}}\Lambda(n)\big([-n^\gamma]-[-(n+1)^\gamma]\big)\Bigg|
                     \nonumber \\
\ll & \,\, L\sum_{n\sim X}n^{\gamma-1}\tau(n-a)+L\sum_{\substack{n\sim X\\ n=[k^{1/\gamma}]}}\tau(n-a)
           \ll X^{\gamma+\frac{\eta}{2}}\ll x^\gamma L^{-A}.
\end{align*}
Therefore, we can assume that $x^{1-\eta}\leqslant X\leqslant x$. It is easy to see that, for $\xi\leqslant(1-\eta)/2$, there holds
\begin{equation*}
 X^\xi\leqslant D\leqslant X^{\xi+\frac{\eta}{2}}.
\end{equation*}
Now, we use the well--known expansions
\begin{equation}\label{psi-expan}
\psi(t)=-\sum_{0<|h|\leqslant H}\frac{e(th)}{2\pi i h}+O(g(t,H)),
\end{equation}
where
\begin{equation*}
g(t,H)=\min\bigg(1,\frac{1}{H\|t\|}\bigg)=\sum_{h=-\infty}^{\infty}b_he(th)
\end{equation*}
and
\begin{equation*}
b_h\ll\min\bigg(\frac{\log2H}{H},\frac{1}{|h|},\frac{H}{|h|^2}\bigg).
\end{equation*}
Putting (\ref{psi-expan}) into the left--hand side of (\ref{suffi-4}), the contribution of the error term in   (\ref{psi-expan}) to the left--hand side of (\ref{suffi-4}) is
\begin{equation}\label{tri-fenjie}
\sum_{\substack{d\leqslant D\\ (a,d)=1}}\sum_{\substack{n\sim X\\ n\equiv a\!\!\!\!\!\pmod d}}\Lambda(n)
\big(g(n^\gamma,H)+g((n+1)^\gamma,H)\big)=R_1+R_2,
\end{equation}
say. We only deal with $R_1$, since the estimate of $R_2$ is exactly the same. For $R_1$, we have
\begin{equation}\label{R_1-upper}
R_1\ll L\sum_{\substack{d\leqslant D\\ (a,d)=1}}\sum_{\substack{n\sim X\\ n\equiv a\!\!\!\!\!\pmod d}}g(n^\gamma,H)
\ll L\sum_{\substack{d\leqslant D\\ (a,d)=1}}\sum_{h=-\infty}^\infty|b_h|
\Bigg|\sum_{\substack{n\sim X\\ n\equiv a\!\!\!\!\!\pmod d}}e(hn^\gamma)\Bigg|.
\end{equation}
Now, we need the following estimate which is an analogue of Lemma 1 of Heath--Brown \cite{Heath-Brown-1983} for
arithmetic progressions.

\begin{lemma}\label{ex-arith-pair}
  Let $1\leqslant d\leqslant X,\, X<X_1\leqslant 2X$. Then
\begin{equation*}
\sum_{\substack{X<n\leqslant X_1\\ n\equiv a\!\!\!\!\!\pmod d}}e(hn^\gamma)
\ll \min\Big(Xd^{-1}, d^{-1}|h|^{-1}X^{1-\gamma}+d^{\kappa-\ell}|h|^\kappa X^{\kappa\gamma-\kappa+\ell}\Big),
\end{equation*}
where $(\kappa,\ell)$ is an exponent pair.
\end{lemma}
\begin{proof}
 We take integer $b$, which satisfies $1\leqslant b\leqslant d$ and $b\equiv a \!\!\pmod d$. Then we derive that
\begin{equation*}
\sum_{\substack{X<n\leqslant X_1\\ n\equiv a\!\!\!\!\!\pmod d}}e(hn^\gamma)=
\sum_{\frac{X-b}{d}<m\leqslant\frac{X_1-b}{d}}e\big(h(b+md)^\gamma\big).
\end{equation*}
Estimating the sum on the right--hand side of above equation trivially and by any exponent pair $(\kappa,\ell)$, we obtain the desired estimate.    $\hfill$
\end{proof}

Taking $(\kappa,\ell)=(\frac{1}{2},\frac{1}{2})$ in Lemma \ref{ex-arith-pair}, we obtain
\begin{align}
  R_1 \ll & \,\, L\sum_{d\leqslant D}\bigg(|b_0|Xd^{-1}+\sum_{h\not=0}|b_h|\Big(|h|^{-1}X^{1-\gamma}d^{-1}+
                 |h|^{1/2}X^{\gamma/2}\Big)\bigg)
                         \nonumber \\
  \ll & \,\,L^3H^{-1}X+LX^{1-\gamma}\sum_{d\leqslant D}d^{-1}\sum_{h\not=0}|h|^{-2}
                          \nonumber \\
      & \,\,+LX^{\gamma/2}D\Bigg(\sum_{0<|h|\leqslant H}|h|^{-1/2}+H\sum_{|h|>H}|h|^{-3/2}\Bigg)
                           \nonumber \\
 \ll & \,\, L^3XH^{-1}+L^2X^{1-\gamma}+LX^{\gamma/2}H^{1/2}D\ll x^\gamma L^{-A},
\end{align}
provided that
\begin{equation}\label{suffi-condi-1}
 H=X^{1-\gamma+\eta}\qquad \textrm{and} \qquad \gamma>\frac{1}{2}+\xi.
\end{equation}
Therefore, it remains to show that
\begin{equation}\label{suffi-S-condi}
 S:=\sum_{\substack{d\leqslant D\\ (a,d)=1}}\sum_{0<h\leqslant H}\frac{1}{h}\Bigg|
 \sum_{\substack{n\sim X\\ n\equiv a\!\!\!\!\!\pmod d}}\Lambda(n)\Big(e\big(-hn^\gamma\big)-e\big(-h(n+1)^\gamma\big)\Big)\Bigg|\ll x^\gamma L^{-A}.
\end{equation}
Set
\begin{equation*}
   \phi_h(n)=1-e\big(h(n^\gamma-(n+1)^\gamma)\big).
\end{equation*}
By partial summation, the innermost sum on the left--hand side of (\ref{suffi-S-condi}) is
\begin{align}\label{inner-sum-upper}
     & \,\, \sum_{\substack{n\sim X\\ n\equiv a\!\!\!\!\!\pmod d}}\Lambda(n)e(-hn^\gamma)\phi_h(n)
                \nonumber \\
   = & \,\,\int_X^{2X}\phi_h(t)\mathrm{d}\Bigg(\sum_{\substack{X<n\leqslant t \\ n\equiv a\!\!\!\!\!\pmod d}}
         \Lambda(n)e(-hn^\gamma)\Bigg)
                \nonumber \\
 \ll &\,\, \big|\phi_h(2X)\big|\Bigg|\sum_{\substack{n\sim X\\ n\equiv a\!\!\!\!\!\pmod d}}\Lambda(n)e(-hn^\gamma)\Bigg|
                \nonumber \\
     &\,\, +\int_X^{2X}\Bigg|\sum_{\substack{X<n\leqslant t \\ n\equiv a\!\!\!\!\!\pmod d}}
         \Lambda(n)e(-hn^\gamma)\Bigg|\bigg|\frac{\partial\phi_h(t)}{\partial t}\bigg|\mathrm{d}t
                \nonumber \\
 \ll &\,\,h X^{\gamma-1}\cdot\max_{X<t\leqslant2X}\Bigg|\sum_{\substack{X<n\leqslant t \\ n\equiv a\!\!\!\!\!\pmod d}}
         \Lambda(n)e(-hn^\gamma)\Bigg|,
\end{align}
where we use the estimate
\begin{equation*}
   \phi_h(t)\ll ht^{\gamma-1}\qquad \textrm{and}\qquad\, \frac{\partial\phi_h(t)}{\partial t}\ll ht^{\gamma-2}.
\end{equation*}
Inserting (\ref{inner-sum-upper}) into the left--hand side of (\ref{suffi-S-condi}), we obtain
\begin{align*}
  S \ll & \,\, X^{\gamma-1}\sum_{\substack{d\leqslant D\\ (a,d)=1}}\sum_{0<h\leqslant H}
               \Bigg|\sum_{\substack{n\sim X\\ n\equiv a\!\!\!\!\!\pmod d}}\Lambda(n)e(-hn^\gamma)\Bigg|
                    \nonumber \\
   = & \,\,X^{\gamma-1}\sum_{\substack{d\leqslant D\\ (a,d)=1}}\sum_{0<h\leqslant H}c(d,h)
           \sum_{\substack{n\sim X\\ n\equiv a\!\!\!\!\!\pmod d}}\Lambda(n)e(-hn^\gamma)
                    \nonumber \\
  \ll & \,\, X^{\gamma-1}\sum_{n\sim X}\Lambda(n)\sum_{0<h\leqslant H}e(-hn^\gamma)
          \sum_{\substack{d\leqslant D\\ (a,d)=1\\ d|n-a}}c(d,h)
                    \nonumber \\
  = & \,\, X^{\gamma-1}\sum_{n\sim X}\Lambda(n)G(n),
\end{align*}
where
\begin{equation*}
  G(n)=\sum_{0<h\leqslant H}\Xi_h(n)e(-hn^\gamma)
\end{equation*}
and
\begin{equation*}
  \Xi_h(n)=\sum_{\substack{d\leqslant D\\ (a,d)=1\\ d|n-a}}c(d,h), \qquad |c(d,h)|=1.
\end{equation*}
Consequently, in order to establish the estimate (\ref{suffi-S-condi}), it is sufficient to show that
\begin{equation}\label{suffi-genera}
  \Bigg|\sum_{n\sim X}\Lambda(n)G(n)\Bigg|\ll XL^{-A}.
\end{equation}
A special case of the identity of Heath--Brown \cite{Heath-Brown-1982} is given by
\begin{equation*}
 -\frac{\zeta'}{\zeta}=-\frac{\zeta'}{\zeta}(1-Z\zeta)^3-\sum_{j=1}^3\binom{3}{j}(-1)^jZ^j\zeta^{j-1}(-\zeta'),
\end{equation*}
where $Z=Z(s)=\sum\limits_{m\leqslant X^{1/3}}\mu(m)m^{-s}$. From this we can decompose $\Lambda(n)$ for $n\sim X$ as
\begin{equation*}
 \Lambda(n)=\sum_{j=1}^3\binom{3}{j}(-1)^{j-1}\sum_{m_1\cdots m_{2j}=n}\mu(m_1)\cdots\mu(m_j)\log m_{2j}.
\end{equation*}
Thus, for any arithmetic function $G(n)$, we can express $\sum\limits_{n\sim X}\Lambda(n)G(n)$ in terms of sums
\begin{equation*}
 \mathop{\sum\,\,\cdots\,\,\sum}_{\substack{m_1\cdots m_{2j}\sim X\\ m_i\sim M_i}}\mu(m_1)\cdots\mu(m_j)
 (\log m_{2j})G(m_1\cdots m_{2j}),
\end{equation*}
where $1\leqslant j\leqslant3,\, M_1M_2\cdots M_{2j}\sim X$ and $M_1,\dots,M_j\leqslant X^{1/3}$. By dividing the $M_j$ into
two groups, we have
\begin{equation}\label{expo-fenjie}
 \Bigg|\sum_{n\sim X}\Lambda(n)G(n)\Bigg|\ll_\eta X^\eta\max\Bigg|\mathop{\sum\sum}_{\substack{mn\sim X\\ m\sim M}}
 a(m)b(n)G(mn)\Bigg|,
\end{equation}
where the maximum is taken over all bilinear forms with coefficients satisfying one of
\begin{equation}\label{type-II-coeff}
   |a(m)|\leqslant 1,\qquad \qquad |b(n)|\leqslant1,
\end{equation}
or
\begin{equation*}
   |a(m)|\leqslant 1,\qquad \qquad b(n)=1,
\end{equation*}
or
\begin{equation*}
   |a(m)|\leqslant 1,\qquad \qquad b(n)=\log n,
\end{equation*}
and also satisfying in all cases
\begin{equation}\label{gene-coeff-condi}
  M\leqslant X.
\end{equation}
We refer to the case (\ref{type-II-coeff}) as being Type II sums and to the other cases as being Type I sums and write for
brevity $\Sigma_{II}$ and $\Sigma_{I}$, respectively. By dividing the $M_j$ into two groups in a judicious fashion we are
able to reduce the range of $M$ from (\ref{gene-coeff-condi}). In Section \ref{ex-section}, we shall
give the estimate of these sums.

In the rest of this section, we shall list several lemmas which is necessary for proving Theorem \ref{Theorem-2}.
\begin{lemma}\label{exponen-fenjie}
If we have real numbers $0<a<1,\,0<b<c<1$  satisfying
\begin{equation*}
b<\frac{2}{3},\qquad 1-c<c-b, \qquad 1-a<\frac{c}{2},
\end{equation*}
then (\ref{expo-fenjie}) still holds when (\ref{gene-coeff-condi}) is replaced by the conditions
\begin{equation*}
M\leqslant X^a \qquad \textrm{for Type I sums},
\end{equation*}
and
\begin{equation*}
X^b\leqslant M\leqslant X^c \qquad \textrm{for Type II sums}.
\end{equation*}
\end{lemma}
\begin{proof}
 See Proposition 1 of Balog and Friedlander \cite{Balog-Friedlander-1992}.  $\hfill$
\end{proof}

\begin{lemma}\label{large-sieve}
For any $Q>1,\,N\geqslant 1$ and any sequence $a(n)$, we have
\begin{equation}
 \sum_{q\sim Q}\frac{1}{\varphi(q)}\,\sideset{}{^*}\sum_{\chi\bmod q}\Bigg|\sum_{n=M+1}^{M+N}a(n)\chi(n)\Bigg|^2
 \ll \bigg(Q+\frac{N}{Q}\bigg)\sum_{n=M+1}^{M+N}|a(n)|^2.
\end{equation}
\end{lemma}
\begin{proof}
 See Theorem 2.11 of Pan and Pan \cite{Pan-Pan-book}.  $\hfill$
\end{proof}

\begin{lemma}\label{fenduan-es}
For  $\frac{1}{2}<\alpha<1,\,J\geqslant1,\,N\geqslant1,\,\Delta>0$, let $\mathscr{N}(\Delta)$ denote the number of solutions of
the following inequality
\begin{equation*}
 \big|h_1n_1^\alpha-h_2n_2^\alpha\big|\leqslant\Delta,\qquad h_1,h_2\sim J,\quad n_1,n_2\sim N.
\end{equation*}
Then we have
\begin{equation*}
 \mathscr{N}(\Delta)\ll \Delta JN^{2-\alpha}+JN\log(JN).
\end{equation*}
\end{lemma}
\begin{proof}
 See the arguments on pp. 256--257 of Heath--Brown \cite{Heath-Brown-1983}.    $\hfill$
\end{proof}

\begin{lemma}\label{char-sum-upper}
For any $A>0$ and non--principal Dirichlet character $\chi\,\!\!\pmod q$ with $q\ll(\log x)^A$, there holds
\begin{equation*}
 \sum_{p\leqslant x}\chi(p)\ll x\exp\Big(-c(A)\sqrt{\log x}\Big),
\end{equation*}
where the implied constant depends only on $A$.
\end{lemma}
\begin{proof}
By partial summation and the arguments on p. 132 of Davenport \cite{Davenport-1980}, it is easy to derive
the desired result.    $\hfill$
\end{proof}

\section{Estimate of Exponential Sums}\label{ex-section}
In this section, we shall give the estimate of exponential sums which will be used in proving Theorem \ref{Theorem-1}.

\subsection{The Estimate of Type II Sums}\label{subse-type-II}
We begin by breaking up the ranges for $n$ and $h$ into intervals $(N,2N]$ and $(J,2J]$ so that $MN\asymp X$ and
$\frac{1}{2}\leqslant J\leqslant H$. Then, for the Type II sums $\Sigma_{II}$, there holds
\begin{equation*}
  \Sigma_{II}\ll L^2\sum_{m\sim M}\Bigg|\sum_{\substack{n\sim N\\ mn\sim X}}\sum_{h\sim J}b(n)\Xi_h(mn)e(h(mn)^\gamma)\Bigg|.
\end{equation*}
Then we have
\begin{equation*}
  0<hn^\gamma\leqslant4JN^\gamma.
\end{equation*}
Denote by $T$ a parameter, which will be chosen later. We decompose the collection of available pairs $(h,n)$ into sets
$\mathscr{S}_t\,(1\leqslant t\leqslant T)$, defined by
\begin{equation*}
  \mathscr{S}_t=\Bigg\{(h,n):\,\,h\sim J,\, n\sim N, \,\frac{4JN^\gamma(t-1)}{T}<hn^\gamma\leqslant
    \frac{4JN^\gamma t}{T}\Bigg\}.
\end{equation*}
Therefore, we have
\begin{equation*}
  \Sigma_{II}\ll L^2\sum_{1\leqslant t\leqslant T}\sum_{m\sim M}\Bigg|\mathop{\sum\sum}_{\substack{(h,n)\in\mathscr{S}_t\\
                 mn\sim X}}b(n)\Xi_h(mn)e(h(mn)^\gamma)\Bigg|,
\end{equation*}
which combines Cauchy's inequality yields
\begin{align*}
& \,\,     |\Sigma_{II}|^2 \ll L^4TM\sum_{1\leqslant t\leqslant T}\sum_{m\sim M}
           \Bigg|\mathop{\sum\sum}_{\substack{(h,n)\in\mathscr{S}_t\\ mn\sim X}}b(n)\Xi_h(mn)e(h(mn)^\gamma)\Bigg|^2
                     \nonumber \\
\ll & \,\, L^4TM\sum_{1\leqslant t\leqslant T}\mathop{\sum\sum}_{(h_1,n_1)\in\mathscr{S}_t}
           \mathop{\sum\sum}_{(h_2,n_2)\in\mathscr{S}_t}\Bigg|\sum_{\substack{m\sim M\\ mn_1\sim X\\ mn_2\sim X}}
           \Xi_{h_1}(mn_1)\Xi_{h_2}(mn_2)e\Big(\big(h_1n_1^\gamma-h_2n_2^\gamma\big)m^\gamma\Big)\Bigg|
                     \nonumber \\
\ll & \,\, L^4TM\mathop{\sum_{h_1\sim J}\sum_{h_2\sim J}\sum_{n_1\sim N}\sum_{n_2\sim N}}_{|\lambda|\leqslant4JN^\gamma T^{-1}}
           \Bigg|\sum_{\substack{m\sim M\\ mn_1\sim X\\ mn_2\sim X}}\Xi_{h_1}(mn_1)\Xi_{h_2}(mn_2)
           e\big(\lambda m^\gamma\big)\Bigg|,
\end{align*}
where
\begin{equation*}
  \lambda=h_1n_1^\gamma-h_2n_2^\gamma.
\end{equation*}
Denote by $\mathcal{S}$ the innermost sum over $m$. First, we use the definition of the quantity $\Xi_h(\cdot)$ and change the order of summation. If the system of the congruences
\begin{equation*}
\begin{cases}
  mn_1\equiv a\!\!\!\pmod {d_1} \\
  mn_2\equiv a\!\!\!\pmod {d_2}
\end{cases}
\end{equation*}
is not solvable, then $\mathcal{S}=0$. If the above system is solvable, then there exists some integer $g=g(n_1,n_2,a,d_1,d_2)$ such that the system is equivalent to $m\equiv g\!\pmod {[d_1,d_2]}$. In this case, we have
\begin{equation*}
\mathcal{S}=\sum_{\substack{d_1\leqslant D\\ (a,d_1)=1}}c(d_1,h_1)\sum_{\substack{d_2\leqslant D\\ (a,d_2)=1}}c(d_2,h_2)
\sum_{\substack{m\sim M\\ mn_1\sim X,\,\,mn_2\sim X\\ m\equiv g\!\!\!\!\!\pmod {[d_1,d_2]}}}e\big(\lambda m^\gamma\big).
\end{equation*}
Therefore, by Lemma \ref{ex-arith-pair} with $(\kappa,\ell)=A^2(\frac{1}{2},\frac{1}{2})=(\frac{1}{14},\frac{11}{14})$,
we obtain
\begin{align*}
              \mathcal{S}
   \ll & \,\, \sum_{\substack{d_1\leqslant D\\ (a,d_1)=1}}\sum_{\substack{d_2\leqslant D\\ (a,d_2)=1}}
              \Bigg|\sum_{\substack{m\sim M\\ mn_1\sim X,\,\,mn_2\sim X\\ m\equiv g\!\!\!\!\!\pmod {[d_1,d_2]}}}
              e\big(\lambda m^\gamma\big)\Bigg|
                 \nonumber \\
   \ll & \,\, \sum_{\substack{d_1\leqslant D\\ (a,d_1)=1}}\sum_{\substack{d_2\leqslant D\\ (a,d_2)=1}}
              \min\Bigg(\frac{M}{[d_1,d_2]},\frac{M^{1-\gamma}}{|\lambda|[d_1,d_2]}+|\lambda|^{\frac{1}{14}}
              [d_1,d_2]^{-\frac{5}{7}}M^{\frac{\gamma}{14}+\frac{5}{7}}\Bigg).
\end{align*}
From the following estimate
\begin{align*}
              \sum_{d_1\leqslant D}\sum_{d_2\leqslant D}[d_1,d_2]^{-\frac{5}{7}}
   \ll & \,\, \sum_{d_1\leqslant D}\sum_{d_2\leqslant D}\bigg(\frac{(d_1,d_2)}{d_1d_2}\bigg)^{\frac{5}{7}}
              =\sum_{1\leqslant r\leqslant D}\sum_{k_1\leqslant\frac{D}{r}}\sum_{k_2\leqslant\frac{D}{r}}
              \frac{1}{r^{5/7}k_1^{5/7}k_2^{5/7}}
                 \nonumber \\
   \ll & \,\, \sum_{1\leqslant r\leqslant D}r^{-\frac{5}{7}}\Bigg(\sum_{k\leqslant\frac{D}{r}}k^{-\frac{5}{7}}\Bigg)^2
              \ll\sum_{1\leqslant r\leqslant D}r^{-\frac{5}{7}}\big(Dr^{-1}\big)^{\frac{4}{7}}\ll D^{\frac{2}{7}},
\end{align*}
we can see that the total contribution of the term $|\lambda|^{\frac{1}{14}}[d_1,d_2]^{-\frac{5}{7}}M^{\frac{\gamma}{14}+\frac{5}{7}}$ to $|\Sigma_{II}|^2$ is
\begin{align}\label{s-upp-1}
 \ll & \,\, |\lambda|^{\frac{1}{14}}M^{\frac{\gamma}{14}+\frac{5}{7}}
            \Bigg(\sum_{d_1\leqslant D}\sum_{d_2\leqslant D}[d_1,d_2]^{-\frac{5}{7}}\Bigg)
            \cdot\mathscr{N}(4JN^\gamma T^{-1})\cdot L^4TM
                   \nonumber \\
\ll & \,\,  \big(JN^\gamma T^{-1}\big)^{\frac{1}{14}}M^{\frac{\gamma}{14}+\frac{5}{7}}D^{\frac{2}{7}}
            \cdot\mathscr{N}(4JN^\gamma T^{-1})\cdot L^4TM
                    \nonumber \\
\ll & \,\,  L^4TM^{\frac{12}{7}}D^{\frac{2}{7}}\big(JM^\gamma N^\gamma T^{-1}\big)^{\frac{1}{14}}
            \cdot\mathscr{N}(4JN^\gamma T^{-1}).
\end{align}
If $|\lambda|\leqslant M^{-\gamma}$, then $M[d_1,d_2]^{-1}\leqslant M^{1-\gamma}|\lambda|^{-1}[d_1,d_2]^{-1}$, and thus the contribution of the $M[d_1,d_2]^{-1}$ term to $|\Sigma_{II}|^2$ is
\begin{equation}\label{S-B-es}
\ll L^4TM\cdot ML^3\cdot \mathscr{N}(M^{-\gamma})\ll L^7TM^2\cdot \mathscr{N}(M^{-\gamma}),
\end{equation}
where we use the elementary estimate
\begin{equation*}
\sum_{d_1\leqslant D}\sum_{d_2\leqslant D}[d_1,d_2]^{-1}\ll (\log D)^3.
\end{equation*}
If $|\lambda|> M^{-\gamma}$, then $M[d_1,d_2]^{-1}> M^{1-\gamma}|\lambda|^{-1}[d_1,d_2]^{-1}$. It follows from the splitting
argument that the contribution of the $M^{1-\gamma}|\lambda|^{-1}[d_1,d_2]^{-1}$ term to $|\Sigma_{II}|^2$ is
\begin{equation}\label{L-B-es}
\ll L^8TM^{2-\gamma}\cdot \max_{M^{-\gamma}\leqslant\Delta\leqslant4JN^\gamma T^{-1}} \mathscr{N}(2\Delta)\Delta^{-1},
\end{equation}
which contains the upper bound estimate (\ref{S-B-es}). From Lemma \ref{fenduan-es}, we know that
\begin{equation*}
\mathscr{N}(\Delta)\ll \Delta JN^{2-\gamma}+JNL,
\end{equation*}
which combines (\ref{s-upp-1}) and (\ref{L-B-es}) yields
\begin{align}\label{Sigma-2-fi-1}
           \big|\Sigma_{II}\big|^2
\ll & \,\, L^4TM^{\frac{12}{7}}D^{\frac{2}{7}}\big(JX^\gamma T^{-1}\big)^{\frac{1}{14}} \cdot\mathscr{N}(4JN^\gamma T^{-1})
                 \nonumber \\
    & \,\, +L^8TM^{2-\gamma}\cdot \max_{M^{-\gamma}\leqslant\Delta\leqslant4JN^\gamma T^{-1}}
           \big(JN^{2-\gamma}+\Delta^{-1}JNL\big)
                 \nonumber \\
\ll & \,\, L^9\Big(M^{-\frac{2}{7}}X^{\frac{\gamma}{14}+2}J^{\frac{29}{14}}D^{\frac{2}{7}}T^{-\frac{1}{14}}
                  +M^{\frac{5}{7}}X^{\frac{\gamma}{14}+1}J^{\frac{15}{14}}D^{\frac{2}{7}}T^{\frac{13}{14}}
                  +TX^{2-\gamma}J+JXMT\Big).
\end{align}
We take $T$ such that the first term and the fourth term in the above estimate are equal. Consequently, we choose
\begin{equation}\label{T-chosen}
T=\big[M^{-\frac{6}{5}}X^{\frac{\gamma+14}{15}}JD^{\frac{4}{15}}\big]+1.
\end{equation}
Putting (\ref{T-chosen}) into (\ref{Sigma-2-fi-1}), we obtain
\begin{align*}
           \big|\Sigma_{II}\big|^2
\ll & \,\, X^\eta\Big(M^{-\frac{2}{5}}X^{\frac{2\gamma+28}{15}}J^2D^{\frac{8}{15}}
                      +M^{-\frac{6}{5}}X^{\frac{44-14\gamma}{15}}J^2D^{\frac{4}{15}}
                      +M^{-\frac{1}{5}}X^{\frac{\gamma+29}{15}}J^2D^{\frac{4}{15}}
                          \nonumber \\
      & \,\,\qquad\quad     +M^{\frac{5}{7}}X^{\frac{\gamma}{14}+1}J^{\frac{15}{14}}D^{\frac{2}{7}}+X^{2-\gamma}J+JXM\Big),
\end{align*}
which combines $J\ll H=X^{1-\gamma+\eta}$ yields
\begin{align*}
           \big|\Sigma_{II}\big|^2
\ll & \,\, X^\eta\Big(M^{-\frac{2}{5}}X^{\frac{58-28\gamma}{15}}D^{\frac{8}{15}}
                      +M^{-\frac{6}{5}}X^{\frac{74-44\gamma}{15}}D^{\frac{4}{15}}
                      +M^{-\frac{1}{5}}X^{\frac{59-29\gamma}{15}}D^{\frac{4}{15}}
                          \nonumber \\
      & \,\,\quad\quad     +M^{\frac{5}{7}}X^{\frac{29}{14}-\gamma}D^{\frac{2}{7}}+X^{3-2\gamma}+MX^{2-\gamma}\Big).
\end{align*}
According to above arguments, we deduce the following lemma.

\begin{lemma}\label{Type-II-es}
Suppose that $\frac{1}{2}<\gamma<1$ and $0<\xi\leqslant(1-\eta)/2$ satisfy the condition
\begin{equation}\label{Type-II-condi}
 \gamma>\max\bigg(\frac{29}{32}+\frac{1}{8}\xi+\eta,\,\,\frac{1}{4}+\xi+\eta\bigg).
\end{equation}
If there holds
\begin{equation*}
 X^{\frac{29(1-\gamma)+4\xi}{3}+\eta}\ll M\ll X^{\gamma-\eta},
\end{equation*}
then we have
\begin{equation*}
 \Sigma_{II}\ll X^{1-\eta}.
\end{equation*}
\end{lemma}

\subsection{The Estimate of Type I Sums}
As in Subsection \ref{subse-type-II}, we begin by breaking up the range for $n$ into intervals $(N,2N]$, such that
$MN\asymp X$. Then according to the definition of the quantity $\Xi_h(\cdot)$, we change the order of summation and derive that \begin{equation}\label{Type-1-fenjie}
 \Sigma_I\ll L^2\sum_{0<h\leqslant H}\mathcal{K}_h,
\end{equation}
where
\begin{equation*}
\mathcal{K}_h=\sum_{\substack{d\leqslant D\\ (a,d)=1}}c(d,h)\sum_{m\sim M}a(m)\sum_{\substack{n\sim N\\ mn\sim X\\ mn\equiv a \!\!\!\!\!\pmod d}}e(h(mn)^\gamma).
\end{equation*}
By Lemma \ref{ex-arith-pair} with exponent pair $(\kappa,\ell)=A^6(\frac{1}{2},\frac{1}{2})=(\frac{1}{254},\frac{247}{254})$, we obtain
\begin{align}\label{K_h-sin}
             \mathcal{K}_h
  \ll & \,\, \sum_{\substack{d\leqslant D\\ (a,d)=1}}\sum_{m\sim M}
             \Bigg|\sum_{\substack{n\sim N\\ mn\sim X\\ mn\equiv a \!\!\!\!\!\pmod d}}e(h(mn)^\gamma)\Bigg|
                   \nonumber \\
  \ll & \,\, \sum_{\substack{d\leqslant D\\ (a,d)=1}}\sum_{m\sim M}\Big(d^{-1}h^{-1}M^{-1}X^{1-\gamma}+
             d^{-\frac{123}{127}}h^{\frac{1}{254}}M^{-\frac{123}{127}}X^{\frac{\gamma}{254}+\frac{123}{127}}\Big)
                    \nonumber \\
  \ll & \,\, h^{-1}X^{1-\gamma}L+h^{\frac{1}{254}}M^{\frac{4}{127}}X^{\frac{\gamma}{254}+\frac{123}{127}}D^{\frac{4}{127}}.
\end{align}
From (\ref{Type-1-fenjie}) and (\ref{K_h-sin}), we have
\begin{align*}
            \Sigma_I
  \ll & \,\, L^4X^{1-\gamma}+L^2H^{\frac{255}{254}}M^{\frac{4}{127}}X^{\frac{\gamma}{254}+\frac{123}{127}}D^{\frac{4}{127}}
                     \nonumber \\
  \ll & \,\, L^4X^{1-\gamma}+X^{\frac{501}{254}-\gamma+\frac{4}{127}\xi+\eta}M^{\frac{4}{127}}.
\end{align*}
According to above estimate, we obtain the following lemma.

\begin{lemma}\label{Type-I-es}
Suppose that $M$ satisfies the condition
\begin{equation*}
 M\ll X^{\frac{127}{4}\gamma-\frac{247}{8}-\xi-\eta}.
\end{equation*}
Then we have
\begin{equation*}
 \Sigma_{I}\ll X^{1-\eta}.
\end{equation*}
\end{lemma}

\section{Proof of Theorem \ref{Theorem-1}}
In this section, we combines the results of Lemma \ref{exponen-fenjie}, Lemma \ref{Type-II-es} and Lemma \ref{Type-I-es} to
complete the proof of Theorem \ref{Theorem-1}.

From Lemma \ref{exponen-fenjie}, we take
\begin{equation*}
 a=\frac{127}{4}\gamma-\frac{247}{8}-\xi-\eta,
\end{equation*}
\begin{equation*}
 b=\frac{29(1-\gamma)+4\xi}{3}+\eta,
\end{equation*}
\begin{equation*}
 c=\gamma-\eta.
\end{equation*}
It is easy to check the conditions (\ref{suffi-condi-1}), (\ref{Type-II-condi}), as well as the inequalities
in Lemma \ref{exponen-fenjie}, hold. Hence we obtain (\ref{suffi-genera}), which suffices to complete the proof of
Theorem \ref{Theorem-1}.

\section{Weighted Sieve and Proof of Theorem \ref{Theorem-2}}
In this section, we shall prove Theorem \ref{Theorem-2} according to the result of Theorem \ref{Theorem-1}, weighted sieve of
Richert, and the method of Chen \cite{Chen-1973}.

\noindent
Let
\begin{equation*}
 \mathscr{A}=\big\{a:a\leqslant x,\, a=p+2,\, p=[k^{1/\gamma}]\big\}.
\end{equation*}
We consider the weighted sum
\begin{equation*}
 W(\mathscr{A},x^{3/32}):=\sum_{\substack{a\in\mathscr{A}\\(a,P(x^{3/32}))=1}}
 \Bigg(1-\lambda\sum_{\substack{x^{3/32}\leqslant p<x^{1/u}\\ p|a}}\bigg(1-\frac{u\log p}{\log x}\bigg)\Bigg),
\end{equation*}
where $u=\xi^{-1}+\varepsilon,\, \lambda=(5-u-\varepsilon)^{-1}$ and
\begin{equation*}
 P(z)=\prod\limits_{2<p<z}p.
\end{equation*}
For convenience, we write
\begin{equation*}
 \mathcal{W}_a=1-\lambda\sum_{\substack{x^{3/32}\leqslant p<x^{1/u}\\ p|a}}\bigg(1-\frac{u\log p}{\log x}\bigg).
\end{equation*}
Then we have
\begin{equation}\label{W-fenjie}
  W(\mathscr{A},x^{3/32})=\sum_{\substack{a\in\mathscr{A}\\ (a,P(x^{3/32}))=1\\ \Omega(a)\leqslant3}}\mathcal{W}_a
   +\sum_{\substack{a\in\mathscr{A}\\ (a,P(x^{3/32}))=1\\ \Omega(a)=4\\ \mu(a)\not=0}}\mathcal{W}_a
   +\sum_{\substack{a\in\mathscr{A}\\ (a,P(x^{3/32}))=1\\ \Omega(a)\geqslant5\\ \mu(a)\not=0}}\mathcal{W}_a
   +\sum_{\substack{a\in\mathscr{A}\\ (a,P(x^{3/32}))=1\\ \Omega(a)\geqslant4\\ \mu(a)=0}}\mathcal{W}_a.
\end{equation}
Obviously, we have
\begin{align}\label{n-s-free>4}
             \sum_{\substack{a\in\mathscr{A}\\ (a,P(x^{3/32}))=1\\ \Omega(a)\geqslant4\\ \mu(a)=0}}\mathcal{W}_a
 \ll  & \,\, \sum_{\substack{a\in\mathscr{A}\\ (a,P(x^{3/32}))=1\\ \mu(a)=0}}\tau(a)
             \ll x^\varepsilon \sum_{x^{3/32}\leqslant p_1\leqslant x^{1/2}}
             \sum_{\substack{p\leqslant x-2\\ p\equiv-2\!\!\!\!\!\pmod {p_1^2}}}
                                 \nonumber \\
 \ll  & \,\, x^\varepsilon\sum_{x^{3/32}\leqslant p_1\leqslant x^{1/2}}\bigg(\frac{x}{p_1^2}+1\bigg)
         \ll  x^\varepsilon\big(x^{1-3/32}+x^{1/2}\big)\ll x^{29/32+\varepsilon}.
\end{align}
For given integer $a$ with $a\leqslant x,(a,P(x^{3/32}))=1$ and $\mu(a)\not=0$, the weight $\mathcal{W}_a$ in the sum $W(\mathscr{A},x^{3/32})$ satisfies
\begin{align}\label{weight-upper}
           1-\lambda\sum_{\substack{x^{3/32}\leqslant p<x^{1/u}\\ p|a}}\bigg(1-\frac{u\log p}{\log x}\bigg)
\leqslant & \,\,\lambda\bigg(\frac{1}{\lambda}-\sum_{p|a}\bigg(1-\frac{u\log p}{\log x}\bigg)\bigg)
                     \nonumber \\
= & \,\, \lambda\bigg(5-u-\varepsilon-\Omega(a)+\frac{u\log a}{\log x}\bigg)<\lambda(5-\Omega(a)),
\end{align}
and thus $\mathcal{W}_a<0$ for $\Omega(a)\geqslant5$. From (\ref{W-fenjie})--(\ref{weight-upper}), we know that
\begin{align}\label{omega(a)<3-lower}
   \sum_{\substack{a\in\mathscr{A}\\ (a,P(x^{3/32}))=1\\ \Omega(a)\leqslant3}}\mathcal{W}_a
= &\,\, W(\mathscr{A},x^{3/32})-\sum_{\substack{a\in\mathscr{A}\\ (a,P(x^{3/32}))=1\\ \Omega(a)=4\\ \mu(a)\not=0}}\mathcal{W}_a
        -\sum_{\substack{a\in\mathscr{A}\\ (a,P(x^{3/32}))=1\\ \Omega(a)\geqslant5\\ \mu(a)\not=0}}\mathcal{W}_a
        +O(x^{29/32+\varepsilon})
                     \nonumber \\
\geqslant &\,\,W(\mathscr{A},x^{3/32})-\sum_{\substack{a\in\mathscr{A}\\ (a,P(x^{3/32}))=1\\ \Omega(a)=4\\ \mu(a)\not=0}}
               \mathcal{W}_a+O(x^{29/32+\varepsilon}).
\end{align}
Therefore, if we can show that the contribution of the second term on the right--hand side of (\ref{omega(a)<3-lower}) is strictly less than $W(\mathscr{A},x^{3/32})$, then we shall prove Theorem \ref{Theorem-2}.

\noindent
For $W(\mathscr{A},x^{3/32})$, we have
\begin{align}\label{W(a)-chai}
         W(\mathscr{A},x^{3/32})
= & \,\, \sum_{\substack{a\in\mathscr{A}\\(a,P(x^{3/32}))=1}}1-\lambda\sum_{x^{3/32}\leqslant p<x^{1/u}}
         \bigg(1-\frac{u\log p}{\log x}\bigg)\sum_{\substack{a\in\mathscr{A}\\(a,P(x^{3/32}))=1\\ p|a}}1
               \nonumber\\
= & \,\, S(\mathscr{A},x^{3/32})-\lambda\sum_{x^{3/32}\leqslant p<x^{1/u}}\bigg(1-\frac{u\log p}{\log x}\bigg)
         S(\mathscr{A}_p,x^{3/32}).
\end{align}
Now, we shall use Theorem 8.4 of Halberstam and Richert \cite{Halberstam-Richert-book} to give the lower bound of
$S(\mathscr{A},x^{3/32})$. Hence in this theorem we take
\begin{equation*}
 X=\pi_\gamma(x),\qquad \omega(d)=
 \begin{cases}
   \displaystyle\frac{d}{\varphi(d)}, & \textrm{if $(d,2)=1$ and $\mu(d)\not=0$},\\
   \,\,\,\,\,0, & \textrm{otherwise}.
 \end{cases}
\end{equation*}
In this section, as usual, let $f(s)$ and $F(s)$ denote the classical functions in the linear sieve
theory. Then by (2.8) and (2.9) of Chapter $8$ in Halberstam and Richert \cite{Halberstam-Richert-book} , we have
\begin{equation*}
  F(s)=\frac{2e^{C_0}}{s},\qquad 0<s\leqslant3;\qquad f(s)=\frac{2e^{C_0}\log(s-1)}{s}, \qquad 2\leqslant s\leqslant4,
\end{equation*}
where $C_0$ denotes Euler's constant. Then it is easy to check the conditions $(\Omega_1)$ and $(\Omega_2(1,L))$ hold. Thus, it is sufficient to show that the condition $(R(1,\alpha))$ holds. Set
\begin{equation*}
 R(x,d):=\sum_{\substack{p\leqslant x\\ p\equiv a\!\!\!\!\!\pmod d \\p=[m^{1/\gamma}]}}1-\frac{1}{\varphi(d)}\pi_\gamma(x),
\end{equation*}
then it follows from Theorem \ref{Theorem-1} that
\begin{equation*}
 \sum_{\substack{d\leqslant x^\xi\\ (d,2)=1}}\big|R(x,d)\big|\ll \frac{x^\gamma}{(\log x)^A}.
\end{equation*}
From the trivial estimate $R(x,d)\ll x^\gamma d^{-1}$ and Cauchy's inequality, we know that
\begin{align*}
    & \,\,   \sum_{\substack{d\leqslant x^\xi\\ (d,2)=1}}\mu^2(d)3^{\nu(d)}\big|R(x,d)\big|
             \ll x^{\gamma/2}\sum_{\substack{d\leqslant x^\xi\\ (d,2)=1}}\frac{\mu^2(d)3^{\nu(d)}}{d^{1/2}}
             \big|R(x,d)\big|^{1/2}
                    \nonumber \\
\ll & \,\,    x^{\gamma/2}\Bigg(\sum_{d\leqslant x^\xi}\frac{\mu^2(d)9^{\nu(d)}}{d}\Bigg)^{1/2}
             \Bigg(\sum_{\substack{d\leqslant x^\xi\\ (d,2)=1}}\big|R(x,d)\big|\Bigg)^{1/2}
                    \nonumber \\
\ll & \,\,    x^{\gamma/2}\Bigg(\mathop{\sum_{d_1}\cdots\sum_{d_9}}_{d_1d_2\cdots d_9\leqslant x^\xi}
             \frac{\mu^2(d_1d_2\cdots d_9)}{d_1d_2\cdots d_9}\Bigg)^{1/2}
             \Bigg(\sum_{\substack{d\leqslant x^\xi\\ (d,2)=1}}\big|R(x,d)\big|\Bigg)^{1/2}
                     \nonumber \\
\ll & \,\,    x^{\gamma/2}\Bigg(\sum_{n\leqslant x^\xi}\frac{1}{n}\Bigg)^{9/2}\Bigg(\frac{x^\gamma}{(\log x)^A}\Bigg)^{1/2}
             \ll \frac{x^\gamma}{(\log x)^A},
\end{align*}
from which we know that the condition $(R(1,\alpha))$ holds. By noting the fact that $2<32\xi/3<4$ holds
for $171/172<\gamma<1$, then Theorem 8.4 of Halberstam and Richert \cite{Halberstam-Richert-book} gives
\begin{align}
                  S(\mathscr{A},x^{3/32})
\geqslant & \,\, \pi_\gamma(x)V(x^{3/32})\Bigg(f\bigg(\frac{32\xi}{3}\bigg)-o(1)\Bigg)
                     \nonumber \\
= & \,\, \frac{3e^{C_0}}{16}\pi_\gamma(x)V(x^{3/32})\Bigg(\frac{\log\big(\frac{32}{3}\xi-1\big)}{\xi}-o(1)\Bigg),
\end{align}
where $C_0$ denotes Euler's constant, and
\begin{equation*}
  V(z)=\prod_{2<p<z}\bigg(1-\frac{\omega(p)}{p}\bigg).
\end{equation*}
Moreover, it follows from (1.11) on p. 245 of Halberstam and Richert \cite{Halberstam-Richert-book} that
\begin{align}\label{A_p-W-lower}
    & \,\, \sum_{x^{3/32}\leqslant p<x^{1/u}}\bigg(1-\frac{u\log p}{\log x}\bigg)S(\mathscr{A}_p,x^{3/32})
                        \nonumber \\
\leqslant & \,\, \pi_\gamma(x)V(x^{3/32})\Bigg(\sum_{x^{3/32}\leqslant p<x^{1/u}}\bigg(1-\frac{u\log p}{\log x}\bigg)
                 \frac{1}{\varphi(p)}F\bigg(\frac{\log(x^\xi/p)}{\log x^{3/32}}\bigg)+o(1)\Bigg)
                         \nonumber \\
= & \,\,  \frac{3e^{C_0}}{16}\pi_\gamma(x)V(x^{3/32})\Bigg(\int_u^{\frac{32}{3}}\frac{\beta-u}{\beta(\xi\beta-1)}
           \mathrm{d}\beta+o(1)\Bigg).
\end{align}
Combining (\ref{W(a)-chai})--(\ref{A_p-W-lower}), we obtain
\begin{equation}\label{W(a,3/32)-lower}
   W(\mathscr{A},x^{3/32})\geqslant\frac{3e^{C_0}}{16}\pi_\gamma(x)V(x^{3/32})
   \Bigg(\frac{\log\big(\frac{32}{3}\xi-1\big)}{\xi}-\lambda\int_u^{\frac{32}{3}}\frac{\beta-u}{\beta(\xi\beta-1)}
   \mathrm{d}\beta+o(1)\Bigg).
\end{equation}
Now, we consider the second term on the right--hand side of (\ref{omega(a)<3-lower}). Set
\begin{equation*}
   \mathcal{B}=\Big\{m:m\leqslant x,\,m=p_1p_2p_3p_4,\, x^{3/32}\leqslant p_1<p_2<p_3<p_4\Big\}.
\end{equation*}
and
\begin{equation*}
   \mathcal{E}=\Big\{n:\,n+2\in\mathcal{B},\,\,n=[k^{1/\gamma}]\Big\}
\end{equation*}
From (\ref{weight-upper}) we deduce that
\begin{align}\label{E-trans-upper}
          \sum_{\substack{a\in\mathscr{A}\\ (a,P(x^{3/32}))=1\\ \Omega(a)=4\\ \mu(a)\not=0}}\mathcal{W}_a
\leqslant & \,\, \lambda \sum_{\substack{a\in\mathscr{A}\\ (a,P(x^{3/32}))=1\\ \Omega(a)=4\\ \mu(a)\not=0}}1
       =\lambda\sum_{\substack{p+2\leqslant x,\,\, p=[k^{1/\gamma}]\\ p+2=p_1p_2p_3p_4\\ x^{3/32}\leqslant p_1<p_2<p_3<p_4}}1
                  \nonumber \\
= & \,\,\lambda\cdot S\big(\mathcal{E},x^{1/2}\big)\leqslant\lambda\cdot S\big(\mathcal{E},x^{\xi/3}\big).
\end{align}
Let $\mathcal{E}_d=\big\{n\in\mathcal{E}: n\equiv0\!\pmod d\big\}$. Then it is easy to see that
\begin{equation*}
   \mathcal{E}_d=\frac{1}{\varphi(d)}\mathcal{X}+\mathscr{R}_d^{(1)}+\mathscr{R}_d^{(2)}+\mathscr{R}_d^{(3)},
\end{equation*}
where
\begin{equation*}
  \mathcal{X}=\sum_{n\in\mathcal{B}}\big((n-1)^\gamma-(n-2)^\gamma\big),
\end{equation*}
\begin{equation}\label{R_d(1)-def}
  \mathscr{R}_d^{(1)}=\sum_{\substack{n\in\mathcal{B}\\n\equiv2\!\!\!\!\!\pmod d}}\big((n-1)^\gamma-(n-2)^\gamma\big)
  -\frac{1}{\varphi(d)}\sum_{\substack{n\in\mathcal{B}\\(n,d)=1}}\big((n-1)^\gamma-(n-2)^\gamma\big),
\end{equation}
\begin{equation}\label{R_d(2)-def}
  \mathscr{R}_d^{(2)}=\sum_{\substack{n\in\mathcal{B}\\n\equiv2\!\!\!\!\!\pmod d}}
  \Big(\psi\big(-(n-1)^\gamma\big)-\psi\big(-(n-2)^\gamma\big)\Big),
\end{equation}
\begin{equation}\label{R_d(3)-def}
  \mathscr{R}_d^{(3)}=
  -\frac{1}{\varphi(d)}\sum_{\substack{n\in\mathcal{B}\\(n,d)>1}}\big((n-1)^\gamma-(n-2)^\gamma\big).
\end{equation}
In order to use Theorem 8.4 of Halberstam and Richert \cite{Halberstam-Richert-book} to give upper bound for $S(\mathcal{E},x^{\xi/3})$, we need to show that
\begin{equation}\label{xi/3-err-condi}
  \sum_{\substack{d\leqslant x^\xi\\ (d,2)=1}}\Big|\mathscr{R}_d^{(i)}\Big|\ll \frac{x^\gamma}{(\log x)^{100}},\qquad i=1,2,3.
\end{equation}
We shall prove (\ref{xi/3-err-condi}) by three following lemmas. For convenience, we put $D=x^\xi$.

\begin{lemma}\label{R_1-err-mean}
Let $\mathscr{R}_d^{(1)}$ be defined as in (\ref{R_d(1)-def}). Then we have
\begin{equation*}
 \sum_{\substack{d\leqslant D\\ (d,2)=1}}\Big|\mathscr{R}_d^{(1)}\Big|\ll \frac{x^\gamma}{(\log x)^{100}}.
\end{equation*}
\end{lemma}
\begin{proof}
 By the orthogonality of Dirichlet characters, the first term in (\ref{R_d(1)-def}) is
\begin{align*}
  & \,\,  \sum_{\substack{n\in\mathcal{B}\\n\equiv2\!\!\!\!\!\pmod d}}\big((n-1)^\gamma-(n-2)^\gamma\big)
              \nonumber \\
= & \,\, \frac{1}{\varphi(d)}\sum_{n\in\mathcal{B}}\big((n-1)^\gamma-(n-2)^\gamma\big)\sum_{\chi\bmod d}
         \chi(n)\overline{\chi(2)}
              \nonumber \\
= & \,\, \frac{1}{\varphi(d)}\sum_{\chi\bmod d}\overline{\chi(2)}
         \sum_{\substack{n\in\mathcal{B}\\ (n,d)=1}}\chi(n)\big((n-1)^\gamma-(n-2)^\gamma\big)
              \nonumber \\
= & \,\, \frac{1}{\varphi(d)}\sum_{\substack{n\in\mathcal{B}\\(n,d)=1}}\big((n-1)^\gamma-(n-2)^\gamma\big)
              \nonumber \\
  & \,\, +\frac{1}{\varphi(d)}\sum_{\substack{\chi\bmod d\\ \chi\not=\chi^0}}\overline{\chi(2)}
          \sum_{\substack{n\in\mathcal{B}\\(n,d)=1}}\chi(n)\big((n-1)^\gamma-(n-2)^\gamma\big).
\end{align*}
Therefore, we have
\begin{align*}
     \sum_{\substack{d\leqslant D\\ (d,2)=1}}\Big|\mathscr{R}_d^{(1)}\Big|
  = & \,\, \sum_{\substack{d\leqslant D\\ (d,2)=1}}\frac{1}{\varphi(d)}
           \Bigg|\sum_{\substack{\chi\bmod d\\ \chi\not=\chi^0}}\overline{\chi(2)}
           \sum_{n\in\mathcal{B}}\chi(n)\big((n-1)^\gamma-(n-2)^\gamma\big)\Bigg|
                 \nonumber \\
  \leqslant & \,\, \sum_{\substack{d\leqslant D\\ (d,2)=1}}\frac{1}{\varphi(d)}
           \sum_{\substack{\chi\bmod d\\ \chi\not=\chi^0}}
           \Bigg|\sum_{n\in\mathcal{B}}\chi(n)\big((n-1)^\gamma-(n-2)^\gamma\big)\Bigg|.
\end{align*}
Let $\chi_q^*$ denote the primitive character which induces $\chi_d$, then we have $1<q|d$ and $\chi_d=\chi_d^0\chi_q^*$.
Consequently, we derive that
\begin{align}\label{induce-ex}
    & \,\, \sum_{\substack{d\leqslant D\\ (d,2)=1}}\frac{1}{\varphi(d)}\sum_{\substack{\chi\bmod d\\ \chi\not=\chi^0}}
           \Bigg|\sum_{n\in\mathcal{B}}\chi(n)\big((n-1)^\gamma-(n-2)^\gamma\big)\Bigg|
                  \nonumber \\
  = & \,\, \sum_{\substack{d\leqslant D\\ (d,2)=1}}\frac{1}{\varphi(d)}\sum_{1<q|d}\,\,\sideset{}{^*}\sum_{\chi\bmod q}
           \Bigg|\sum_{n\in\mathcal{B}}\chi_d^0(n)\chi_q^*(n)\big((n-1)^\gamma-(n-2)^\gamma\big)\Bigg|
                  \nonumber \\
  = & \,\, \sum_{2<q\leqslant D}\,\,\sideset{}{^*}\sum_{\chi\bmod q}\sum_{\substack{d\leqslant D\\
           (d,2)=1\\ d\equiv0\!\!\!\!\!\pmod q}}\frac{1}{\varphi(d)}
           \Bigg|\sum_{n\in\mathcal{B}}\chi_d^0(n)\chi_q^*(n)\big((n-1)^\gamma-(n-2)^\gamma\big)\Bigg|
                  \nonumber \\
\ll & \,\, \sum_{2<q\leqslant D}\,\,\sideset{}{^*}\sum_{\chi\bmod q}\sum_{\substack{d\leqslant D\\
           (d,2)=1\\ d\equiv0\!\!\!\!\!\pmod q}}\frac{1}{\varphi(d)}
           \Bigg|\sum_{n\in\mathcal{B}}\chi_d^0(n)\chi_q^*(n)\gamma n^{\gamma-1}\Bigg|
                   \nonumber \\
    & \,\, +\sum_{2<q\leqslant D}\,\,\sideset{}{^*}\sum_{\chi\bmod q}\sum_{\substack{d\leqslant D\\
           (d,2)=1\\ d\equiv0\!\!\!\!\!\pmod q}}\frac{1}{\varphi(d)}
           \Bigg|\sum_{n\in\mathcal{B}}\chi_d^0(n)\chi_q^*(n)\big((n-1)^\gamma-(n-2)^\gamma-\gamma n^{\gamma-1}\big)\Bigg|.
\end{align}
The second term on the right--hand side of (\ref{induce-ex}) can be estimated as
\begin{align*}
  \ll & \,\, \sum_{2<q\leqslant D}\,\,\sideset{}{^*}\sum_{\chi\bmod q}\sum_{\substack{d\leqslant D\\ q|d}}\frac{1}{\varphi(d)}
             \sum_{n\leqslant x}n^{\gamma-2}
  \ll  x^{\gamma-1}\sum_{q\leqslant D}\,\,\sideset{}{^*}\sum_{\chi\bmod q}\sum_{d_1\leqslant D/q}\frac{1}{\varphi(d_1q)}
                  \nonumber \\
  \ll & \,\, x^{\gamma-1}\sum_{q\leqslant D}\,\,\sideset{}{^*}\sum_{\chi\bmod q}\sum_{d_1\leqslant D/q}
              \frac{1}{\varphi(q)\varphi(d_1)}
  \ll  x^{\gamma-1}D\log D=x^{\gamma-1+\xi}\log D\ll x^{\gamma}(\log x)^{-A}.
\end{align*}
Hence, it is sufficient to show that, for $1\leqslant Q\leqslant D$, there holds
\begin{equation}\label{R_1-suufi}
\sum_{q\sim Q}\,\,\sideset{}{^*}\sum_{\chi\bmod q}\sum_{\substack{d\leqslant D\\(d,2)=1\\ d\equiv0\!\!\!\!\!\pmod q}}
\frac{1}{\varphi(d)}\Bigg|\sum_{n\in\mathcal{B}}\chi_d^0(n)\chi_q^*(n) n^{\gamma-1}\Bigg|
 \ll \frac{x^\gamma}{(\log x)^{100}}.
\end{equation}
Next, we shall prove (\ref{R_1-suufi}) in two cases.

\noindent
\textbf{Case 1} If $Q\leqslant(\log x)^{300}$, by the definition of $\mathcal{B}$, partial summation and Lemma \ref{char-sum-upper}, we deduce that
\begin{align}\label{case-1-inner}
   & \,\,   \Bigg|\sum_{n\in\mathcal{B}}\chi_d^0(n)\chi_q^*(n) n^{\gamma-1}\Bigg|
                   \nonumber \\
 = & \,\, \Bigg|\sum_{x^{3/32}\leqslant p_1<p_2<p_3<(x/(p_1p_2))^{1/2}}
    (p_1p_2p_3)^{\gamma-1}\chi_d^0(p_1p_2p_3)\chi_q^*(p_1p_2p_3)
                   \nonumber \\
  & \,\,\quad \quad \times \sum_{p_3<p_4<x/(p_1p_2p_3)}\chi_d^0(p_4)\chi_q^*(p_4)p_4^{\gamma-1}\Bigg|
                   \nonumber \\
  \ll & \,\,\sum_{x^{3/32}\leqslant p_1<p_2<p_3<(x/(p_1p_2))^{1/2}}(p_1p_2p_3)^{\gamma-1}
            \Bigg(\Bigg|\sum_{p_3<p_4<x/(p_1p_2p_3)}\chi_q^*(p_4)p_4^{\gamma-1}\Bigg|+O(1)\Bigg)
                   \nonumber \\
  \ll & \,\,\sum_{x^{3/32}\leqslant p_1<p_2<p_3<(x/(p_1p_2))^{1/2}}(p_1p_2p_3)^{\gamma-1}
            \bigg(\frac{x}{p_1p_2p_3}\bigg)^\gamma\cdot\exp\big(-c_1\log^{1/2}x\big)
                     \nonumber \\
  \ll & \,\, x^\gamma\exp\big(-\log^{1/3}x\big).
\end{align}
Putting (\ref{case-1-inner}) into (\ref{R_1-suufi}), we obtain
\begin{align*}
  & \,\, \sum_{q\sim Q}\,\,\sideset{}{^*}\sum_{\chi\bmod q}\sum_{\substack{d\leqslant D\\(d,2)=1\\ d\equiv0\!\!\!\!\!\pmod q}}
         \frac{1}{\varphi(d)}\Bigg|\sum_{n\in\mathcal{B}}\chi_d^0(n)\chi_q^*(n) n^{\gamma-1}\Bigg|
                    \nonumber \\
  \ll & \,\, x^\gamma\exp\big(-\log^{1/3}x\big)\sum_{q\sim Q}\,\,\sideset{}{^*}\sum_{\chi\bmod q}\sum_{d_1\leqslant D/q}
             \frac{1}{\varphi(d_1q)}
                    \nonumber \\
  \ll & \,\, x^\gamma\exp\big(-\log^{1/3}x\big)\sum_{q\sim Q}\,\,\sideset{}{^*}\sum_{\chi\bmod q}\sum_{d_1\leqslant D/q}
             \frac{1}{\varphi(q)\varphi(d_1)}
                    \nonumber \\
  \ll & \,\, x^\gamma\exp\big(-\log^{1/3}x\big)Q\log D\ll x^\gamma\exp\big(-\log^{1/4}x\big).
\end{align*}

\noindent
\textbf{Case 2} If $(\log x)^{300}\leqslant Q\leqslant D$, by a splitting argument, it is sufficient to show that
\begin{equation}\label{R_2-suufi}
\sum_{q\sim Q}\,\,\sideset{}{^*}\sum_{\chi\bmod q}\sum_{\substack{d\leqslant D\\(d,2)=1\\ d\equiv0\!\!\!\!\!\pmod q}}
\frac{1}{\varphi(d)}\Bigg|\sum_{\substack{n\in\mathcal{B}\\ x/2<n\leqslant x}}\chi_d^0(n)\chi_q^*(n) n^{\gamma-1}\Bigg|
 \ll \frac{x^\gamma}{(\log x)^{120}}.
\end{equation}
From splitting argument, it is easy to see that the innermost sum in (\ref{R_2-suufi}) can be represented as the sum of at most
$O(\log^4x)$ sums of the form
\begin{equation*}
  \sum_{\substack{p_j\in\mathcal{I}_j\\ j=1,2,3,4}}\chi_d^0\big(p_1p_2p_3p_4\big)\chi_q^*\big(p_1p_2p_3p_4\big)\big(p_1p_2p_3p_4\big)^{\gamma-1},
\end{equation*}
where
\begin{align*}
  &  \mathcal{I}_j=\big(N_j,N_j'\,\big],\qquad x^{3/32}\leqslant N_j<N_j'\leqslant2N_j,\qquad j=1,2,3,4, \\
  &  x/2<N_1N_2N_3N_4<N_1'N_2'N_3'N_4'\leqslant x.
\end{align*}
Set
\begin{align*}
  g_j(\chi_q^*)=\sum_{p\in\mathcal{I}_j}\chi_q^*(p)p^{\gamma-1},\qquad j=1,2,3,4.
\end{align*}
Then we have
\begin{align*}
  \sum_{p\in\mathcal{I}_j}\chi_d^0(p)\chi_q^*(p)p^{\gamma-1}=g_j(\chi_q^*)+O(1), \qquad j=1,2,3,4,
\end{align*}
and thus
\begin{align}\label{case-2-inner-trans}
   & \,\, \sum_{\substack{p_j\in\mathcal{I}_j\\ j=1,2,3,4}}
          \chi_d^0\big(p_1p_2p_3p_4\big)\chi_q^*\big(p_1p_2p_3p_4\big)\big(p_1p_2p_3p_4\big)^{\gamma-1}
                \nonumber \\
\ll &\,\, \prod_{i=1}^4\Big(g_i(\chi_q^*)+O(1)\Big)\ll \sum_{i=1}^4\sum_{1\leqslant j_1<\cdots<j_i\leqslant4}
          \prod_{k=1}^i\big|g_{j_k}\big(\chi_q^*\big)\big|+O(1).
\end{align}
Trivially, we have the elementary estimate
\begin{align*}
  \sum_{\substack{d\leqslant D\\ q|d}}\frac{1}{\varphi(d)}=\sum_{d_1\leqslant D/q}\frac{1}{\varphi(d_1q)}\ll
  \sum_{d_1\leqslant D/q}\frac{1}{\varphi(d_1)\varphi(q)}\ll\frac{\log D}{\varphi(q)}.
\end{align*}
From the above estimate and (\ref{case-2-inner-trans}), it is easy to see that, in order to prove (\ref{R_2-suufi}), we only
need to prove
\begin{equation*}
  \Sigma:=\sum_{\substack{q\sim Q\\ (q,2)=1}}\frac{1}{\varphi(q)}\,\,\sideset{}{^*}\sum_{\chi\bmod q}
          \Bigg|\prod_{k=1}^ig_{j_k}(\chi_q^*)\Bigg|\ll x^\gamma(\log x)^{-150}
\end{equation*}
with $1\leqslant i\leqslant 4,\, 1\leqslant j_1<\cdots<j_i\leqslant4$ and $(\log x)^{300}\leqslant Q\leqslant D$.

\noindent
Set
\begin{equation*}
  \mathscr{F}_1(\chi_q^*)=g_{j_1}(\chi_q^*)=\sum_{M<m\leqslant2M}a(m)\chi_q^*(m)m^{\gamma-1}
\end{equation*}
and
\begin{equation*}
  \mathscr{F}_2(\chi_q^*)=\prod_{k=2}^ig_{j_k}(\chi_q^*)=\sum_{N<n\leqslant2^{i-1}N}b(n)\chi_q^*(n)n^{\gamma-1},
\end{equation*}
where
\begin{equation*}
  M,N\ll x^{29/32},\qquad MN\asymp x,\qquad a(m)\ll1,\qquad b(n)\ll 1.
\end{equation*}
It follows from Cauchy's inequality and Lemma \ref{large-sieve} that
\begin{align*}
  \Sigma = & \,\, \sum_{\substack{q\sim Q\\ (q,2)=1}}\frac{1}{\varphi(q)}\,\,\sideset{}{^*}\sum_{\chi\bmod q}
                   \Big|\mathscr{F}_1(\chi_q^*)\mathscr{F}_2(\chi_q^*)\Big|
                        \nonumber \\
  \ll & \,\, \sum_{\substack{q\sim Q\\ (q,2)=1}}\frac{1}{\varphi(q)}
             \Bigg(\,\,\,\sideset{}{^*}\sum_{\chi\bmod q}\Big|\mathscr{F}_1(\chi_q^*)\Big|^2\Bigg)^{1/2}
             \Bigg(\,\,\,\sideset{}{^*}\sum_{\chi\bmod q}\Big|\mathscr{F}_2(\chi_q^*)\Big|^2\Bigg)^{1/2}
                        \nonumber \\
  \ll & \,\, \Bigg(\sum_{\substack{q\sim Q\\ (q,2)=1}}\frac{1}{\varphi(q)}\,\,\sideset{}{^*}\sum_{\chi\bmod q}
             \Big|\mathscr{F}_1(\chi_q^*)\Big|^2\Bigg)^{1/2}
             \Bigg(\sum_{\substack{q\sim Q\\ (q,2)=1}}\frac{1}{\varphi(q)}\,\,\sideset{}{^*}\sum_{\chi\bmod q}
             \Big|\mathscr{F}_2(\chi_q^*)\Big|^2\Bigg)^{1/2}
                        \nonumber \\
  \ll & \,\, \Bigg(\bigg(Q+\frac{M}{Q}\bigg)\bigg(\sum_{m\sim M}\Big|a(m)m^{\gamma-1}\Big|^2\bigg)\Bigg)^{1/2}
             \Bigg(\bigg(Q+\frac{N}{Q}\bigg)\bigg(\sum_{n\sim N}\Big|b(n)n^{\gamma-1}\Big|^2\bigg)\Bigg)^{1/2}
                        \nonumber \\
  \ll & \,\, (MN)^{(2\gamma-1)/2}\bigg(Q+M^{1/2}+N^{1/2}+\frac{x^{1/2}}{Q^{1/2}}\bigg)
                        \nonumber \\
  \ll & \,\, x^{\gamma-1/2}\Big(D+x^{29/64}+x^{1/2}(\log x)^{-150}\Big)\ll x^\gamma(\log x)^{-150}.
\end{align*}
Combining the results of Case $1$ and Case $2$, we derive the desired result. $\hfill$
\end{proof}

\begin{lemma}\label{R_2-err-mean}
Let $\mathscr{R}_d^{(2)}$ be defined as in (\ref{R_d(2)-def}). Then we have
\begin{equation*}
 \sum_{\substack{d\leqslant D\\ (d,2)=1}}\Big|\mathscr{R}_d^{(2)}\Big|\ll \frac{x^\gamma}{(\log x)^{100}}.
\end{equation*}
\end{lemma}
\begin{proof}
From the definition of $\mathscr{R}_d^{(2)}$, it suffices to show that, $X\leqslant x$, there holds
\begin{equation}\label{R_2-suffi-condi}
 \sum_{\substack{d\leqslant D\\ (d,2)=1}}\Bigg|\sum_{\substack{n\in\mathcal{B}\\ n\sim X\\ n\equiv2\!\!\!\!\!\pmod d}}
 \Big(\psi\big(-(n-1)^\gamma\big)-\psi\big(-(n-2)^\gamma\big)\Big)\Bigg|\ll \frac{x^\gamma}{(\log x)^{A}}.
\end{equation}
If $X\leqslant x^{1-\eta}$, then the left--hand side of (\ref{R_2-suffi-condi}) is
\begin{align*}
\ll & \,\, \sum_{\substack{d\leqslant D\\ (d,2)=1}}\Bigg|\sum_{\substack{n\sim X\\ n\equiv 2\!\!\!\!\!\pmod d}}
           \big((n-1)^\gamma-(n-2)^\gamma\big)\Bigg|
                   \nonumber \\
   & \,\,  +\sum_{\substack{d\leqslant D\\ (d,2)=1}}
           \Bigg|\sum_{\substack{n\sim X\\ n\equiv 2\!\!\!\!\!\pmod d}}\big([-(n-2)^\gamma]-[-(n-1)^\gamma]\big)\Bigg|
                     \nonumber \\
\ll & \,\, \sum_{n\sim X}(n-1)^{\gamma-1}\tau(n-2)+L\sum_{\substack{n\sim X\\ n=[k^{1/\gamma}]}}\tau(n-2)
           \ll X^{\gamma+\frac{\eta}{2}}\ll x^\gamma(\log x)^{-A}.
\end{align*}
Now, we assume that $x^{1-\eta}<X\leqslant x$, by (\ref{psi-expan}). We know that the total contribution of the error term in   (\ref{psi-expan}) to the left--hand side of (\ref{R_2-suffi-condi}) is
\begin{equation*}
\sum_{\substack{d\leqslant D\\ (d,2)=1}}\sum_{\substack{n\sim X\\ n\equiv 2\!\!\!\!\!\pmod d}}
\big(g((n-1)^\gamma,H)+g((n-2)^\gamma,H)\big),
\end{equation*}
which can be treated as (\ref{tri-fenjie})--(\ref{suffi-condi-1}), and we get
\begin{equation*}
\sum_{\substack{d\leqslant D\\ (d,2)=1}}\sum_{\substack{n\sim X\\ n\equiv 2\!\!\!\!\!\pmod d}}
\big(g((n-1)^\gamma,H)+g((n-2)^\gamma,H)\big)\ll x^\gamma(\log x)^{-A}.
\end{equation*}
The contribution of the main term in (\ref{psi-expan}) to the left--hand side of (\ref{R_2-suffi-condi}) is
\begin{align}\label{case-cf}
   =&\,\, \sum_{\substack{d\leqslant D\\ (d,2)=1}}\Bigg|\sum_{\substack{n\in\mathcal{B}\\ n\sim X\\
           n\equiv 2\!\!\!\!\!\pmod d}} \sum_{0<h\leqslant H}\frac{e(-h(n-2)^\gamma)-e(-h(n-1)^\gamma)}{2\pi ih}\Bigg|
                \nonumber \\
   \ll &\,\,\sum_{\substack{d\leqslant D\\ (d,2)=1}}\sum_{0<h\leqslant H}\frac{1}{h}\Bigg|\sum_{\substack{n\in\mathcal{B}\\
            n\sim X\\n\equiv 2\!\!\!\!\!\pmod d}}\Big(e\big(-hn^\gamma\big)-e\big(-h(n+1)^\gamma\big)\Big) \Bigg|
                \nonumber \\
   &\,\, +\sum_{\substack{d\leqslant D\\ (d,2)=1}}\sum_{0<h\leqslant H}\frac{1}{h}\Bigg|\sum_{\substack{n\in\mathcal{B}\\
            n\sim X\\n\equiv 2\!\!\!\!\!\pmod d}}
            \bigg(\Big(e(-h(n-2)^\gamma)-e(-h(n-1)^\gamma)\Big)
                \nonumber \\
    &\,\,   \qquad\qquad\qquad\qquad\quad -\Big(e\big(-hn^\gamma\big)-e\big(-h(n+1)^\gamma\big)\Big)\bigg) \Bigg|.
\end{align}
The second term on the right-hand side of (\ref{case-cf}) can be estimated as
\begin{equation*}
  \ll \sum_{\substack{d\leqslant D\\ (d,2)=1}}\sum_{0<h\leqslant H}\frac{1}{h}\sum_{n\sim X}hn^{\gamma-2}\ll
  HDX^{\gamma-1}\ll X^{\xi+\eta}\ll x^{\gamma}(\log x)^{-A}.
\end{equation*}
Consequently, it suffices to show that
\begin{equation*}
\mathfrak{S}:=\sum_{\substack{d\leqslant D\\ (d,2)=1}}\sum_{0<h\leqslant H}\frac{1}{h}\Bigg|\sum_{\substack{n\in\mathcal{B}\\
n\sim X\\n\equiv 2\!\!\!\!\!\pmod d}}\Big(e\big(-hn^\gamma\big)-e\big(-h(n+1)^\gamma\big)\Big) \Bigg|
\ll x^\gamma (\log x)^{-A}.
\end{equation*}
Define
\begin{equation*}
\mathfrak{f}_h(\ell)=1-e\big(h(\ell^\gamma-(\ell+1)^\gamma)\big).
\end{equation*}
It follows from partial summation that
\begin{align*}
          \mathfrak{S}
 = & \,\, \sum_{\substack{d\leqslant D\\ (d,2)=1}}\sum_{0<h\leqslant H}
          \frac{1}{h}\Bigg|\sum_{\substack{\ell\in\mathcal{B}\\ \ell\sim X\\ \ell\equiv 2\!\!\!\!\!\pmod d}}
          e\big(-h\ell^\gamma\big)\mathfrak{f}_h(\ell) \Bigg|
               \nonumber \\
 = & \,\, \sum_{\substack{d\leqslant D\\ (d,2)=1}}\sum_{0<h\leqslant H}\frac{1}{h}\Bigg|\int_X^{2X}\mathfrak{f}_h(u)\mathrm{d}
          \Bigg(\sum_{\substack{\ell\in\mathcal{B}\\ \ell\equiv 2\!\!\!\!\!\pmod d\\ X<\ell\leqslant u}}
          e\big(-h\ell^\gamma\big)\Bigg)\Bigg|
               \nonumber \\
 \ll & \,\, \sum_{\substack{d\leqslant D\\ (d,2)=1}}\sum_{0<h\leqslant H}\frac{1}{h}\Bigg(\Big|\mathfrak{f}_h(2X)\Big|
            \Bigg|\sum_{\substack{\ell\in\mathcal{B}\\ \ell\sim X\\ \ell\equiv 2\!\!\!\!\!\pmod d}}
            e\big(-h\ell^\gamma\big)\Bigg|
               \nonumber \\
   & \,\, \qquad \qquad+\int_X^{2X}\Bigg|\sum_{\substack{\ell\in\mathcal{B}\\ \ell\equiv 2\!\!\!\!\!\pmod d\\
          X<\ell\leqslant u}}e\big(-h\ell^\gamma\big)\Bigg|\bigg|\frac{\partial \mathfrak{f}_h(u)}{\partial u}
          \bigg|\mathrm{d}u\Bigg)
               \nonumber \\
 \ll & \,\, X^{\gamma-1}\sum_{\substack{d\leqslant D\\ (d,2)=1}}\sum_{0<h\leqslant H}\max_{X<u\leqslant2X}
            \Bigg|\sum_{\substack{\ell\in\mathcal{B}\\ \ell\equiv 2\!\!\!\!\!\pmod d\\
            X<\ell\leqslant u}}e\big(-h\ell^\gamma\big)\Bigg|,
\end{align*}
where we use the estimate
\begin{equation*}
  \big|\mathfrak{f}_h(u)\big|\ll hu^{\gamma-1}\qquad \textrm{and}\qquad\,\,
  \bigg|\frac{\partial \mathfrak{f}_h(u)}{\partial u}\bigg|\ll hu^{\gamma-2}.
\end{equation*}
There, we obtain
\begin{align*}
            \sum_{\substack{d\leqslant D\\ (d,2)=1}}\Big|\mathscr{R}_d^{(2)}\Big|
 \ll & \,\, x^\gamma(\log x)^{-A}+\max_{\substack{X<u\leqslant2X\\ x^{1-\eta}<X\leqslant x}}X^{\gamma-1}
             \sum_{\substack{d\leqslant D\\ (d,2)=1}}\sum_{0<h\leqslant H}
             \sum_{\substack{\ell\in\mathcal{B}\\ \ell\equiv 2\!\!\!\!\!\pmod d\\X<\ell\leqslant u}}\alpha(d,h)e(-h\ell^\gamma)
                      \nonumber \\
 \ll & \,\, x^\gamma(\log x)^{-A}+\max_{\substack{X<u\leqslant2X\\ x^{1-\eta}<X\leqslant x}}X^{\gamma-1}
            \Bigg|\sum_{\substack{\ell\in\mathcal{B}\\\ell\sim X}}\sum_{0<h\leqslant H}\Theta_h(\ell)e(-h\ell^\gamma)\Bigg|,
\end{align*}
where
\begin{equation*}
  \Theta_h(\ell)=\sum_{\substack{d\leqslant D\\ (d,2)=1\\ d|\ell-2}}\alpha(d,h),\qquad \quad \big|\alpha(d,h)\big|=1.
\end{equation*}
Next, we shall illustrate that, for $\ell=p_1p_2p_3p_4\in\mathcal{B}$ and $\ell\sim X>x^{1-\eta}$, there must be some partial
product of $p_1p_2p_3p_4$ which lies in the interval $[X^{1/2+\eta},X^{85/86-\eta}]$.

First, since $p_i\geqslant x^{3/32}$ and $p_1p_2p_3p_4\in[x^{1-\eta},x]$, we have $p_i\leqslant X^{85/86-\eta}$. If there exists some $p_i\in[X^{1/2+\eta},X^{85/86-\eta}]$, then the conclusion follows. If this case does not exist, we consider the
product $p_1p_2$. At this time, there must be $p_1p_2<X^{1/2+\eta}$. Otherwise, from $p_1p_2\geqslant X^{1/2+\eta}>(x^{1-\eta})^{1/2+\eta}>x^{1/2}$ we obtain $p_3p_4=n(p_1p_2)^{-1}<x^{1/2}<p_1p_2$, which contradict to $p_1<p_2<p_3<p_4$. Now, we consider the product $p_1p_2p_3$. If $p_1p_2p_3\in[X^{1/2+\eta},X^{85/86-\eta}]$, then the conclusion holds. Otherwise, if $p_1p_2p_3<X^{1/2+\eta}$, then $p_4=n(p_1p_2p_3)^{-1}>X(X^{1/2+\eta})^{-1}=X^{1/2-\eta}>x^{7/16}$, and thus $p_1p_2p_4>x^{6/32+7/16}=x^{5/8}>X^{1/2+\eta}$.
Moreover, $p_1p_2p_4=n(p_3)^{-1}\leqslant x^{29/32}\leqslant X^{85/86-\eta}$. Above all, there must exist some partial product
of $p_1p_2p_3p_4$ which lies in $[X^{1/2+\eta},X^{85/86-\eta}]$.

For $85/86<\gamma<1$ and the definition of $\xi$, it is easy to see that
\begin{equation*}
  X^{\frac{29(1-\gamma)+4\xi}{3}+\eta}\leqslant X^{\frac{1}{2}+\eta}<X^{\frac{85}{86}-\eta}\leqslant X^{\gamma-\eta},
\end{equation*}
which combines Lemma \ref{Type-II-es} yields
\begin{equation*}
\sum_{\substack{d\leqslant D\\ (d,2)=1}}\Big|\mathscr{R}_d^{(2)}\Big|\ll x^\gamma(\log x)^{-A}.
\end{equation*}
This completes the proof of  Lemma \ref{R_2-err-mean}. $\hfill$
\end{proof}

\begin{lemma}\label{R_3-err-mean}
Let $\mathscr{R}_d^{(3)}$ be defined as in (\ref{R_d(3)-def}). Then we have
\begin{equation*}
 \sum_{\substack{d\leqslant D\\ (d,2)=1}}\Big|\mathscr{R}_d^{(3)}\Big|\ll \frac{x^\gamma}{(\log x)^{A}}.
\end{equation*}
\end{lemma}
\begin{proof}
We have
\begin{equation*}
   \mathscr{R}_d^{(3)}=-\frac{1}{\varphi(d)}\sum_{\substack{n\in\mathcal{B}\\ (n,d)>1}}\gamma n^{\gamma-1}
   -\frac{1}{\varphi(d)}\sum_{\substack{n\in\mathcal{B}\\ (n,d)>1}}\Big(\big((n-1)^\gamma-(n-2)^\gamma\big)-\gamma n^{\gamma-1}\Big).
\end{equation*}
Hence
\begin{equation}\label{R_3-upper-1}
   \sum_{\substack{d\leqslant D\\ (d,2)=1}}\Big|\mathscr{R}_d^{(3)}\Big|\ll
   \sum_{d\leqslant D}\frac{1}{\varphi(d)}\sum_{\substack{n\in\mathcal{B}\\ (n,d)>1}} n^{\gamma-1}+
   \sum_{d\leqslant D}\frac{1}{\varphi(d)}\sum_{\substack{n\in\mathcal{B}\\ (n,d)>1}}
   \Big(\big((n-1)^\gamma-(n-2)^\gamma\big)-\gamma n^{\gamma-1}\Big).
\end{equation}
The second term on the right--hand side of (\ref{R_3-upper-1}) is
\begin{equation}\label{R_3-upper-1-nd}
  \ll \sum_{d\leqslant D}\frac{1}{\varphi(d)}\sum_{n\leqslant x}n^{\gamma-2}\ll x^{\gamma-1}
      \sum_{d\leqslant D}\frac{1}{\varphi(d)}\ll x^{\gamma-1+\eta}\ll x^\gamma(\log x)^{-A}.
\end{equation}
For the first term, which is on the right--hand side of (\ref{R_3-upper-1}), we have
\begin{align}\label{R_3-upper-1-st}
  \ll & \,\, \sum_{d\leqslant D}\frac{1}{\varphi(d)}\sum_{\substack{n\in\mathcal{B}\\ (n,d)\geqslant x^{3/32}}}n^{\gamma-1}
             \ll x^{\gamma+\eta}\mathop{\sum_{d\leqslant D}\sum_{n\leqslant x}}_{(n,d)\geqslant x^{3/32}}\frac{1}{nd}
 \ll  x^{\gamma+\eta}\sum_{x^{3/32}\leqslant k\leqslant D}\mathop{\sum_{d\leqslant D}\sum_{n\leqslant x}}_{(n,d)=k}
            \frac{1}{nd}
                \nonumber \\
 \ll& \,\, x^{\gamma+\eta}\sum_{x^{3/32}\leqslant k\leqslant D}\sum_{n_1\leqslant D/k}\sum_{d_1\leqslant D/k}
            \frac{1}{k^2d_1n_1}\ll x^{\gamma+2\eta}\sum_{x^{3/32}\leqslant k\leqslant D}\frac{1}{k^2}
            \ll x^{\gamma-3/32+\eta}.
\end{align}
Combining (\ref{R_3-upper-1})--(\ref{R_3-upper-1-st}), we derive the desired result of Lemma \ref{R_3-err-mean}. $\hfill$
\end{proof}

From Lemma \ref{R_1-err-mean}--\ref{R_3-err-mean}, we deduce that
\begin{equation}\label{S(E)-upper}
S(\mathcal{E},x^{\xi/3})\leqslant \mathcal{X}V(x^{\xi/3})\big(F(3)+o(1)\big).
\end{equation}
By Theorem 7.11 of Pan and Pan \cite{Pan-Pan-book}, we know that
\begin{equation}\label{V(z)-expan}
V(z)=C(\omega)\frac{e^{-C_0}}{\log z}\bigg(1+O\bigg(\frac{1}{\log x}\bigg)\bigg),
\end{equation}
where $C_0$ is Euler's constant and $C(\omega)$ is a convergent infinite product defined by
\begin{equation*}
C(\omega)=\prod_p\bigg(1-\frac{\omega(p)}{p}\bigg)\bigg(1-\frac{1}{p}\bigg)^{-1}.
\end{equation*}
According to (\ref{V(z)-expan}), we get
\begin{equation*}
V(x^{\xi/3})=\frac{9}{32\xi}V(x^{3/32})\big(1+O(\log x)^{-1}\big),
\end{equation*}
from which and (\ref{S(E)-upper}) we deduce that
\begin{equation}\label{S(E)-upper-1}
S(\mathcal{E},x^{\xi/3})\leqslant \frac{3e^{C_0}}{16\xi}\mathcal{X}V(x^{3/32})(1+o(1)).
\end{equation}
Next, we compute the quantity $\mathcal{X}$ definitely. Obviously, we have
\begin{equation}\label{X-expan}
\mathcal{X}=\sum_{n\in\mathcal{B}}\gamma n^{\gamma-1}+\sum_{n\in\mathcal{B}}\Big((n-1)^\gamma-(n-2)^\gamma
            -\gamma n^{\gamma-1}\Big).
\end{equation}
For the second term in (\ref{X-expan}), we have
\begin{align}\label{X-expan-nd-upper}
  & \,\,   \sum_{n\in\mathcal{B}}\Big((n-1)^\gamma-(n-2)^\gamma-\gamma n^{\gamma-1}\Big)
        \ll \sum_{n\in\mathcal{B}}n^{\gamma-2}
              \nonumber \\
 \ll& \,\, \bigg(\sum_{x^{3/32}\leqslant p\leqslant x}p^{\gamma-2}\bigg)^4
 \ll  \bigg(\sum_{x^{3/32}\leqslant m\leqslant x}m^{\gamma-2}\bigg)^4\ll x^{3(\gamma-1)/8}=o(1).
\end{align}
For the first term in (\ref{X-expan}), we have
\begin{align}\label{X-first-num}
   & \,\, \sum_{n\in\mathcal{B}}\gamma n^{\gamma-1}
                \nonumber \\
 = & \,\, \gamma \sum_{x^{3/32}\leqslant p_1< x^{1/4}}\sum_{p_1<p_2<(x/p_1)^{1/3}}
          \sum_{p_2<p_3<(x/(p_1p_2))^{1/2}}\sum_{p_3<p_4<x/(p_1p_2p_3)}(p_1p_2p_3p_4)^{\gamma-1}
                \nonumber \\
 = & \,\, \gamma\big(1+o(1)\big)\int_{x^{3/32}}^{x^{1/4}}\int_{u_1}^{(\frac{x}{u_1})^{1/3}}
          \int_{u_2}^{(\frac{x}{u_1u_2})^{1/2}}\int_{u_3}^{\frac{x}{u_1u_2u_3}}
          \frac{(u_1u_2u_3u_4)^{\gamma-1}\mathrm{d}u_4\mathrm{d}u_3\mathrm{d}u_2\mathrm{d}u_1}
          {(\log u_1)(\log u_2)(\log u_3)(\log u_4)}
                 \nonumber \\
 = & \,\, \gamma\big(1+o(1)\big)\int_{\frac{3}{32}}^{\frac{1}{4}}\frac{\mathrm{d}t_1}{t_1}
          \int_{t_1}^{\frac{1-t_1}{3}}\frac{\mathrm{d}t_2}{t_2}
          \int_{t_2}^{\frac{1-t_1-t_2}{2}}\frac{\mathrm{d}t_3}{t_3}
          \int_{t_3}^{1-t_1-t_2-t_3}\frac{x^{(t_1+t_2+t_3+t_4)\gamma}}{t_4}\mathrm{d}t_4.
\end{align}
For the innermost integral in (\ref{X-first-num}), we have
\begin{align}\label{inner-expli}
       \int_{t_3}^{1-t_1-t_2-t_3}\frac{x^{(t_1+t_2+t_3+t_4)\gamma}}{t_4}\mathrm{d}t_4
   = & \,\, \frac{1}{\gamma \log x}\int_{t_3}^{1-t_1-t_2-t_3}\frac{1}{t_4}\mathrm{d}x^{(t_1+t_2+t_3+t_4)\gamma}
                  \nonumber \\
   = & \,\, \frac{1}{\gamma \log x}\bigg(\frac{x^\gamma}{1-t_1-t_2-t_3}+O\bigg(\frac{x^\gamma}{\log x}\bigg)\bigg)
                   \nonumber \\
   = & \,\, \frac{1}{1-t_1-t_2-t_3}\cdot\frac{x^\gamma}{\gamma\log x}(1+o(1)).
\end{align}
From (\ref{X-expan-nd-upper})--(\ref{inner-expli}), we deduce that
\begin{equation}\label{X-compu-num}
 \mathcal{X}=\frac{x^\gamma(1+o(1))}{\log x}\int_{\frac{3}{32}}^{\frac{1}{4}}\frac{\mathrm{d}t_1}{t_1}
          \int_{t_1}^{\frac{1-t_1}{3}}\frac{\mathrm{d}t_2}{t_2}
          \int_{t_2}^{\frac{1-t_1-t_2}{2}}\frac{\mathrm{d}t_3}{t_3(1-t_1-t_2-t_3)}.
\end{equation}
Combining (\ref{omega(a)<3-lower}), (\ref{W(a,3/32)-lower}), (\ref{E-trans-upper}), (\ref{S(E)-upper-1}) and (\ref{X-compu-num}), we obtain
\begin{align*}
                  \sum_{\substack{a\in\mathscr{A}\\(a,P(x^{3/32}))=1\\ \Omega(a)\leqslant3}}\mathcal{W}_a
 \geqslant & \,\,  \frac{3e^{C_0}}{16}\frac{x^\gamma}{\log x}V(x^{3/32})(1+o(1))
                  \Bigg(\frac{\log\big(\frac{32}{3}\xi-1\big)}{\xi}-
                  \lambda\int_u^{\frac{32}{3}}\frac{\beta-u}{\beta(\xi\beta-1)}\mathrm{d}\beta
                            \nonumber \\
 & \,\, -\frac{\lambda}{\xi}\int_{\frac{3}{32}}^{\frac{1}{4}}\frac{\mathrm{d}t_1}{t_1}
                   \int_{t_1}^{\frac{1-t_1}{3}}\frac{\mathrm{d}t_2}{t_2}
                   \int_{t_2}^{\frac{1-t_1-t_2}{2}}\frac{\mathrm{d}t_3}{t_3(1-t_1-t_2-t_3)}\Bigg)
                   +O(x^{\frac{3}{32}+\varepsilon}).
\end{align*}
By simple numerical calculations, it is easy to see that the number in the above brackets $(\cdot)$ is $\geqslant0.000060486$,
provided that $0.9989445<\gamma<1$. This completes the proof of Theorem \ref{Theorem-2}.

\section*{Acknowledgement}


The authors would like to express the most sincere gratitude to the referee
for his/her patience in refereeing this paper. This work is supported by the
Fundamental Research Funds for the Central Universities (Grant No. 2019QS02),
 and National Natural Science Foundation of China (Grant No. 11901566, 11971476).

\end{document}